\documentclass{amsart}
\usepackage{amsmath, amssymb,standard,graphicx,tikz, zref}
\newcommand{\sslash}{\mathbin{/\mkern-6mu/}}

\usepackage{enumitem}
\usepackage{hyperref}
\begin{document}

\title{Partially wrapped Fukaya categories of simplicial skeleta}


\author[L. Katzarkov]{Ludmil Katzarkov}
\address{Department of Mathematics, University of Miami, Coral Gables, FL, 33146, USA 
	\indent National Research University Higher School of Economics, Moscow, Russia, 101000
	\indent Fakult\"at f\"ur Mathematik , Universit\"at Wien, 1090 Wien, Austria}
\email{lkatzark@math.uci.edu}


\author[G. Kerr]{Gabriel Kerr}
\address{Department of Mathematics, Kansas State University, Manhattan, KS, 66506, USA}
\email{gdkerr@math.ksu.edu}

\begin{abstract}
A class of partially wrapped Fukaya categories in $T^* N$ are proven to be well defined and then studied. In the case of $N$ diffeomorphic to $\mathbb{R}^m \times \mathbb{T}^n$, it is shown that these categories provide homological mirrors to equivariant and non-equivariant categories of coherent sheaves on toric varieties. This gives a new Floer theoretic formulation of mirror symmetry in the non-complete case.
\end{abstract}

\maketitle

\section{Introduction}

This paper defines and studies a partially wrapped Fukaya category associated to a Hamiltonian $H$ on an exact symplectic manifold $M$.  Outside of the Floer theoretic context, there have been topological and sheaf theoretic counterparts of such categories introduced in \cite{nadler16, stz}. In particular, the constructible sheaf approach has been well studied following \cite{fltz, nz}. For certain types of Fukaya categories, the topological version has been shown to be equivalent to a Fukaya category defined via Floer theory. For example, when the Lagrangian involved is perturbed infinitesimally at infinity as in \cite{nz}. However, when the perturbations at infinity are asymptotic to a wrapping Hamiltonian which is infinitesimal in certain directions and quadratic in others, the symplectic Fukaya category, known as a partially wrapped Fukaya category, has not been shown to be equivalent to their topological counterpart. There have been several recent works with symplectic formulations of such partially wrapped Fukaya categories both in the dimension $1$ case \cite{aaeko, hkk, lee} and in higher dimensions \cite{auroux, sylvan}.  This paper, in contrast to these referenced works, will consider defining the partially wrapped category directly as in the wrapped setting, but with a relaxed assumption on the properness of the wrapping Hamiltonian.  Attempting such a direct definition comes with a price, and several assumptions must be imposed in order to have a well defined category. Nevertheless, we obtain a general criterion for which a partially wrapped Fukaya category $\wrapped{M}{H}$ is well defined.

We then apply this construction to the case where $H$ arises from a generally non-compact Lagrangian skeleton $Z \subset M$ with well behaved singularities. Here the Hamiltonian $H$ is a smooth perturbation of the kinetic energy function (which is half the distance from $Z$ squared). The retraction $ \pi: M \to Z$ induced by the (metric) gradient of $H$ has smooth Lagrangian fibers over smooth points of $Z$. Thus we may consider the subcategory $\mathcal{W}_H^Z (M)$ of  $\mathcal{W}_H (M)$ generated by such Lagrangians. With a manageable generating collection  in hand, the structure of this category can be approached first using a local description and second by applying a sheaf theoretic framework. A more topological approach to the sheaf theory is given in \cite{gs16}.

To obtain the local description of $\mathcal{W}_H^Z (M)$,  the structure of the Lagrangian skeleton near a singularity $p \in Z$ is described by a simplicial complex $K$. This type of singularity is not as general as the arboreal singularities of \cite{nadler16}, but it nonetheless arises naturally for the examples we consider later in the paper. To such a singularity, we examine in detail the local partially wrapped category $\wrsm{K}$ which is modeled on the inverse image $U_p = \pi^{-1} (V_p)$ where $V_p$ is a neighborhood of $p$ in $Z$. This category can be fully characterized  by the combinatorial structure of $K$. To pass from $\wrsm{K}$ to the global description $\mathcal{W}_H^Z (M)$, one assembles these categories into a cosheaf on $Z$ whose global sections equal $\mathcal{W}_H^Z (M)$.

As an application of this machinery, we consider homological mirror symmetry for a smooth toric variety $X_\Sigma$. The key point to emphasize in this application is that the attribute ``projective'' is conspicuously absent. In particular, among our toric varieties will be affine varieties such as $\mathbb{A}^n$ and, more importantly, quasi-projective varieties such as $\mathbb{A}^n \backslash \{0\}$. Generally, homological mirror symmetry for toric varieties can generally be divided into two flavors. One half concerns the equivalence between the Fukaya category, or $A$-model, of $X_\Sigma$  and the category of matrix factorizations for a mirror superpotential; while the other half concerns a conjectural equivalence
\begin{align} \label{eq:hms}
D^b (X_\Sigma ) \cong \mathcal{F} (X_\Sigma^\vee ).
\end{align}
Here, the category $\mathcal{F} (X_\Sigma^\vee )$ takes a variety of forms, depending on the properties satisfied by $X_\Sigma$ and the researchers' tastes and goals (e.g. \cite{abouzaid09, ako,  fltz, givental, hv}). Using the construction of \cite{fltz}, we consider the mirror skeleton $Z$ in the cotangent bundle of a torus. This skeleton is the prototype for Lagrangian with  singularities dictated by a simplicial complex. We thus conclude with a proof that its partially wrapped Fukaya category $\mathcal{W}_H^{Z} ( U_\Sigma )$, where $U_\Sigma$ is a neighborhood of the skeleton $Z$, satisfies the equivalence of homological mirror symmetry in equation~\eqref{eq:hms}. 

As mentioned above, a major advantage of this approach over the afore mentioned proofs is that $X_\Sigma$ need not be complete in order for the equality to hold. In many such cases a Fukaya category mirror, which lies properly within the purview of symplectic geometry, is conspicuously absent. For example, the homological mirror of $X_\Sigma = \mathbb{A}^2 - \{0\}$ does not have a description as the Fukaya-Seidel category of a Landau-Ginzburg model potential. Thus, a consequence of this work is to establish a symplectic $A$-model mirror category to $D^b (X_\Sigma )$. 

More importantly, all smooth complete toric varieties are elementary examples of git quotients $X \sslash G$. Such quotients are obtained by choosing an equivariant line bundle $\mathcal{L}$ and taking the geometric quotient of the quasi-projective variety $X^{\tiny{ss}}$ of semi-stable points. In order to formulate homological mirror symmetry of such quotients, it is of great benefit do have $A$-model mirrors of all varieties involved $X$, $X^{\tiny{ss}}$ and $X \sslash G$. The partially wrapped perspective gives a single prescription for symplectic mirrors to each variety in this setting, opening the door to an investigation of mirrors symmetry of git quotients and their variations.

{\em Acknowledgements:} 
Support for the first author was provided by Laboratory of Mirror Symmetry NRU HSE, RF Government grant, ag. No. 14.641.31.0001 and by Simons Collaboration Award in Homological Mirror Symmetry. The second author was provided support by the Simons Collaboration Grant. We would like to thank Denis Auroux, Matthew Ballard, Colin Diemer, David Favero, Fabian Haiden, Paul Horja, Tony Pantev, Paul Seidel and Zach Sylvan for helpful comments and conversation during the preparation of this work. We especially express our appreciation for the insights and suggestions made by Maxim Kontsevich which heavily influenced this work.

\section{\label{sec:fukaya}Partially wrapped Fukaya categories}

In this section, we define a partially wrapped Fukaya category for a manifold with codimension $2$ corners. We impose basic conditions on our manifold so that the standard arguments of \cite{abouzaid12,as,fss} for compactness of moduli space can be imported to this setting.

\subsection{Preliminary definitions}
In this subsection we recall the essential definitions for a partial wrapped Fukaya category. We adopt many notations and conventions from \cite{as}, with a key difference being the lack of properness of the wrapping Hamiltonian. The central obstruction to extending the machinery of wrapped Fukaya categories to this setting is proving Gromov compactifications of the required moduli spaces exist. This is addressed by considering an auxiliary function, which yields a
maximum principle argument, in concert with a restriction on the type of Lagrangians we admit into our category. We begin by recalling the essential notation of Floer theory for wrapped categories.

Let $D$ be the unit disc with $(d + 1)$ marked points $\{\zeta^0,
\zeta^1, \ldots, \zeta^d\}$ oriented counter-clockwise. We will say that $S = D \backslash \{\zeta^0, \ldots, \zeta^d\}$ is a \textit{pointed disc} and write $\partial_k S$ for the component of the boundary between $\zeta^k$ and $\zeta^{k + 1}$. We assume a choice $\mathbf{\epsilon} = \{\epsilon^0, \ldots, \epsilon^d\}$ of strip-like ends has been made. These are holomorphic embeddings, $\epsilon^i : \mathbb{R}_{>0} \times [0,1] \to S$ for $i \ne 0$ while $\epsilon^0 : \mathbb{R}_{< 0} \times [0,1] \to S$,  $\epsilon^i ({\mathbb{R} \times \{0,1\} } ) \subset \partial S$
and $\lim_{s \to \pm \infty} \epsilon^i ( s, \cdot ) = \zeta^i$. We also choose a set of weights $\mathbf{w} = \{w^0, \ldots, w^d\} \subset \mathbb{Z}_{> 0}$ satisfying $w_0 = \sum_{i = 1}^d w_i$. Finally, we choose a $1$-form $\gamma \in \Omega^1 (S)$ which satisfies 
\begin{align}
\label{eq:gamma1} \tn \gamma & \leq 0 , \\
\label{eq:gamma2} \gamma|_{\partial S} & = 0, \\
\label{eq:gamma3} (\epsilon^i)^* (\gamma) & = w^i \tn t \text{ for } |s| \gg 0 .
\end{align}
The last condition makes the behavior of $\gamma$ rigid at the strip-like ends so that, relative to the parametrization
$\epsilon^i$, it is constant in the $t$ direction. The moduli space of conformal structures on $D$ with marked points will be denoted $\mathcal{R}^{d + 1}$ and its stable compactification, a manifold with corners, will be denoted $\overline{\mathcal{R}}^{d + 1}$. For $p \in \mathcal{R}^{d + 1}$ we denote by $j_p$ a complex structure on $D$ associated to $p$. 

We will assume that $1$-forms $\{\gamma_p\}_{p \in \mathcal{R}^{d + 1}}$ have been chosen  smoothly over these moduli spaces  so that they are compatible with the gluing and they satisfy conditions \eqref{eq:gamma1}, \eqref{eq:gamma2} and \eqref{eq:gamma3}.  We will furthermore assume that a global boundedness condition on our $1$-forms as follows. Use the standard symplectic structure on $D$ and the complex structure $j_p$ associated to $p \in \mathcal{R}^{d + 1}$ to obtain the metric $\| \|_{j_p}$ on $T D$. Then we assume there exists a bound $L$ such that 
\begin{align} \label{eq:gamma4} \|\gamma_p \|_{j_p} & \leq L , & \| \tn \gamma_p \|_{j_p} & \leq L ,
\end{align}
for all $p \in \mathcal{R}^{d + 1}$. 

For the target space, we take an exact symplectic manifold $M$ with exact $1$-form $\lambda$ and write $\mathcal{J} (M), \mathcal{J}^{int} (M), \mathcal{J}_\omega^c (M)$ and $\mathcal{J}_\omega^t (M)$ for the space of smooth, integrable, compatible and tame almost
complex structures, respectively. For $J \in \mathcal{J} (M)$, a smooth function $h : M \to \mathbb{R}$ is $J$-convex if the $2$-form 
\[ \omega_h = - \tn (\tn^c h ) \] is strictly positive. Here $\tn^c h = \tn h \circ J$ and positivity means $\omega_h (v, J v) > 0$ for all non-zero $v \in TM$. In particular, this implies that $\omega_h$ is a symplectic form and $J$ is tame with respect to $\omega_h$. We denote the induced ($J$-invariant) metric $ 1/2 (\omega_h (v_1, Jv_2) + \omega_h (v_2, J v_1))$ by $\left< v_1, v_2 \right>_h$ and the norm by $\| \|_h$. If $J$ is integrable then we generally have $J \in \mathcal{J}_{\omega_h}^c (M)$.  In more generality, the compatibility of $J$ with respect to $\omega_h$ is inextricably linked to a type of partial integrability of $J$ relative to $h$. This can be written precisely by considering the anti-symmetric part $\eta_h (v_1, v_2) =  1/2 (\omega_h (v_1, Jv_2) - \omega_h (v_2, J v_1))$  of $\omega_h (v_1, Jv_2)$ and computing
\begin{align} \label{eq:antisymmetric} \eta_h (v_1, v_2) = - \frac{1}{2} (J N_J (v_1, v_2)) (h)  \end{align}
where $N_J$ is the Nijenhuis tensor of $J$.

We equip $M$ with  $J \in \mathcal{J} (M)$ and fix a Hamiltonian $H : M \to \mathbb{R}$ which will later satisfy additional properties discussed in Section \ref{subsection:convexity}. For now, we require only the following. 

\begin{definition}
	We will say $(H,J)$ is \textit{partial convex data} for $M$ if 
	\begin{enumerate}
		\item $H$ is bounded below,
		\item $H$ is $J$-convex outside of a compact set in $M$,
		\item there exists $C > 0$ such that $- \tn^c H = \lambda$ for $H (m ) > C$, 
		\item the Liouville vector field $Z_h$ and its negative $-Z_h$ are complete.
	\end{enumerate}
\end{definition}
In the usual setting of convex symplectic manifolds, one assumes $H$ to be an exhaustion function, rendering the fourth condition unnecessary. The key distinction in our setup being that $h$ will generally have non-compact level and sublevel sets. 

To achieve transversality, we will need to perturb our partial convex data on compact subsets. Fixing a bounded below
Hamiltonian $H : M \to \mathbb{R}$, we consider continuous families $\mathbf{H} = \{(h_p , J_p) \}_{p \in X}$ where $X$ is a topological space.
\begin{definition} Let $H : M \to \mathbb{R}$ and $J \in \mathcal{J} (M)$. A continuous family $\{(h_p ,
	J_p) \}_{p \in X}$ of partial convex data has compact support relative to $(H, J)$ if 
	\[ H(m) - h_p (m) ,  \hspace{1cm} J (m) - J_p (m) \] 
	have compact support in $M \times X$. Denote by $\mathcal{H}_{X, H, J}$, or $\mathcal{H}_X$ if $H$ and $J$ are fixed, the space of all such families. 
\end{definition}

Having fixed a wrapping Hamiltonian function $H : M \to \mathbb{R}$, take a decorated pointed disc $(S, \mathbf{\epsilon}, \mathbf{w}, \gamma )$ with associated complex structure $j$ and consider a family $\mathbf{H}=\{(H_p, J_p)\} \in \mathcal{H}_S$. Pairing $\mathbf{H}$ with $\gamma$ gives perturbation data for Floer's equation.  More explicitly, let $E = M \times S$ be the trivial bundle over $S$ and
write $\pi_S $, $\pi_M$ for the projection of $E$ to $S$ and $M$ respectively. We employ Gromov's
trick and define an almost complex structure $I_{\mathbf{H}} \in \mathcal{J} (E)$. Specifically,
for any $(m, p) \in E$ and $(v, w) \in T_{(m,p)} E$ we define
\begin{equation} \label{eq:gromovtrick}
\left. I_{\mathbf{H}} \right|_{(m,p)} (v, w)= (J_p v -  \gamma_p (w) J_p X_{H_p} + \gamma_p (jw) X_{H_p} 
, j w),
\end{equation}
where $X_{H_p}$ is the Hamiltonian vector field associated to $H_p$. A section
$\sigma \in \Gamma (S, E)$ is then $I_\mathbf{H}$-holomorphic if it satisfies
$\overline{\partial}_{I_\mathbf{H}} \sigma = 0$. Projecting to $M$, we observe that $u :=
\pi_M \circ \sigma$ satisfies the equation
\begin{equation}
J_p \tn u (w) - \gamma_p (w) J_p X_{H_p} + \gamma_p (jw)  X_{H_p} = \tn u (j
w).
\end{equation}
This is equivalent to Floer's equation 
\begin{equation} \label{eq:Floer}
(\tn u - X_H \otimes \gamma )^{0, 1} = 0
\end{equation}
which is utilized for the Floer homology of $H$ and, more related to our purpose, wrapped Fukaya categories (see, for example, equation (3.18) of \cite{as}). In particular, counts of the solutions to Floer's equations are used to define the coefficients in the $A_\infty$-structure maps for the Fukaya category. Before discussing these maps, we must describe the boundary conditions for  equation \eqref{eq:Floer}, or equivalently, the objects in our category. Their existence depends on several topological prerequisites which $M$ must satisfy, as well as decorations associated to $M$. We now briefly recall these conditions, leaving a more detailed exposition to previously mentioned references.

We assume that $2 c_1 (M) = 0$ and choose a trivialization  $\Omega \in
\Gamma \left( \bigwedge_\mathbb{C}^n T^* M \right)^{\otimes 2}$. For a Lagrangian $L \subset M$, we will assume that its Maslov class $\mu_L$ vanishes so that there exists a lift $\grade_L : L \to \R$ of $\arg (\Omega (TL^{\otimes 2})) : L \to S^1$. By definition, the choice of lift  $\grade_L$ makes $L$ into a graded Lagrangian. For two graded Lagrangians $L_0, L_1$ intersecting transversely at $p$, we can assign 
\begin{equation*}\deg (p) = \left\lfloor \frac{1}{\pi} (\grade_{L_1}(p) - \grade_{L_0} (p)) \right\rfloor .
\end{equation*}
In addition, we assume $w_2 (L) = 0$ and choose a Pin structure $P^\#$ on $L$.
Again, given $p \in L_0 \cap L_1$, these choices determine a one-dimensional
real vector space $o_{p}$ (see \cite[Section 11]{seidel}). Using the
orientations of $o_p$, one can canonically define a vector space
$|o_p|_{\mathbb{K}}$ over any field $\mathbb{K}$ so that reversing orientation has the effect of multiplying by $-1$. This allows us to orient our moduli spaces and obtain a category over a field $\mathbb{K}$ for which $\textnormal{char} (\mathbb{K} ) \ne 2$. 


Given a Hamiltonian $H$, we associate an object of $\wrapped{M}{H}$ to each Lagrangian $L \subset M$ which satisfies the following properties
\begin{align}
\label{eq:lagpr1} & L \cap H^{-1}((-\infty, C]) \text{ is compact for any } C > 0, \\
\label{eq:lagpr2} & \lambda |_L = \tn f \text{ and } f|_{L \cap H^{-1} (C, \infty)} = 0 \text{ for some } C > 0.
\end{align}
The second condition is typical of general wrapped Fukaya
categories, while the first condition would be unnecessary had we assumed $H$ was proper. In the partial setting, this condition is very stringent and eliminates all Lagrangians on which $H$ is not exhaustive. A Lagrangian brane $L^\#$ relative to $H$ will be the triple $(L, \delta_L, P^\#)$ where $L$ is a  Lagrangian satisfying these properties with a grading $\delta_L$ and a Pin structure $P^\#$. We associate an object $L^\# \in \wrapped{M}{H}$ to each Lagrangian brane $L^\#$. 

Recall that the morphisms in the wrapped setting are obtained by formal sums of flow trajectories of $X_H$ from
one Lagrangian to another. Given $w \in \mathbb{Z}_{>0}$, we write
$\mathcal{X}_w$ to be the space of paths $x : [0, 1] \to M $ satisfying the flow equation $x^\prime (p) = w X_H (x(p))$. In other words, if $\phi^t : M \to M$ is the time $t$ flow of $X_H$, $x \in \mathcal{X}_w$ if and only if $x (t) = \phi^{wt} (m)$ for some $m \in M$.

A pair of objects $L_0^\#, L_1^\#$, we will be said to $H$-transverse if
$\phi^t (L_0)$ transversally intersects $L_1$ for all $t \in \R_{>0}$. Given such a pair, take 
\begin{align} \label{eq:defflow} \mathcal{X}_w (L_0, L_1) = \{x \in \mathcal{X}_w : x(i) \in
L_i \} \end{align}
and write $\deg (x)$ for the degree of the intersection $x(1) \in \psi^w
(L_0) \cap L_1$. Also, write  $|o_x|_{\mathbb{K}}$ for the
$\mathbb{K}$-normalized orientation space associated to $x(1)$. 
With these structures in place, we define morphisms as the graded vector space 
\begin{equation}
\text{Hom}^\bullet (L_0^\#, L_1^\#) = \lim_{w \to \infty}  \oplus_{x \in
	\mathcal{X}_w (L_0, L_1)} |o_x|_{\mathbb{K}} [-\deg (x)].
\end{equation}
Here the limit is taken with respect to the continuation map of
$H$. From this, we follow the prescription for the $A_\infty$-structure maps
given in \cite{as}. Namely, given $H$-transverse objects $L_0^\#, \ldots, L_d^\#$ and
a choice of weights $\mathbf{w}$ consider paths $\mathbf{x} := \{x^0, \ldots, x^d\}$ where $x^k \in \mathcal{X}_{w^k} (L_k, L_{k + 1})$. Recall the
definition of the moduli space $\mathcal{R}^{d + 1, \mathbf{w}} (\mathbf{x})$ of stable popsicle maps \cite[Section~3c]{as}. An element is a pair $(S, u)$ where $S$ is a decorated pointed
disc $(S, \mathbf{\epsilon}, \mathbf{w}, \gamma )$ and $u : S \to M$ satisfies 
\begin{align}
u (\partial_k S) & \subset L_k , \\
\label{eq:disc2} \lim_{s \to \pm \infty} u (\epsilon^k (s, \_)) & = x^k , \\
\bar{\partial}_I \sigma & = 0.
\end{align}
Here $\sigma$ is the unique section of $E$ defined by $u$ and $I$ is defined in
equation \eqref{eq:gromovtrick}.  The virtual dimension of $\mathcal{R}^{d + 1,
	\mathbf{w}} (\mathbf{x} )$ is $\text{vdim} (\mathcal{R}^{d + 1, \mathbf{w}}
(\mathbf{x} )) = d - 2 + \deg (x_0) - \sum_{k = 1}^d \deg (x_k)$. 

For each point $(S, u) \in \mathcal{R}^{d + 1,
	\mathbf{w}} (\mathbf{x} )$, we obtain an isomorphism
\begin{equation}
|o_{(S,u)}| : |o_{x^d}|_{\mathbb{K}} \otimes \cdots \otimes
|o_{x^1}|_{\mathbb{K}} \to |o_{x^0}|_{\mathbb{K}} .
\end{equation}
When $\text{vdim} (\mathcal{R}^{d + 1, \mathbf{w}} (\mathbf{x} )) = 0$, if Gromov compactness holds, we may sum together these contributions to obtain the map
\begin{equation}
\text{m}^{d, \mathbf{w}} = \sum_{(S,u)} |o_{(S,u)}|.
\end{equation}
These in turn are used to define the multiplication map
\begin{equation}
\mu^d  (x^d \otimes \cdots \otimes x^1) := \lim_{\mathbf{w} \to \infty} \oplus
\text{m}^{d, \mathbf{w}} (x^d \otimes \cdots \otimes x^1).
\end{equation}
The constructions reviewed above are now standard in the definition of wrapped
Fukaya categories, and the results on transversality and signs all carry over
without difficulty to the partially wrapped setting. However, for the
$A_\infty$-maps to be well defined and the $A_\infty$ quadratic relation to hold, one requires Gromov compactness theorems of the moduli spaces $\mathcal{R}^{d + 1, \mathbf{w}} (\mathbf{x})$. The following section establishes a criterion which will be used to verify this property.

\subsection{Partial convexity} \label{subsection:convexity}

Given a family $\mathbf{H} = \{(H_p, J_p)\}_{p \in S} \in \mathcal{H}_S$, the  almost complex structure ${I_\mathbf{H}}$ defined in equation \eqref{eq:gromovtrick} paves the way
for the next definition. Consider a partial compactification $\bar{E} = M \times D \subset M \times S$ obtained by
adding the marked points $\{\zeta_0, \ldots, \zeta_{d + 1} \}$. We will say that $\mathbf{H}$ extends to
$\bar{E}$ if it has a unique smooth extension. Under this assumption, for any $C \in \R$ we take 
\begin{equation} \bar{E}_{C} = \{(m, p) \in \bar{E} : p \in D , H_p (m) \leq C \}.
\end{equation}
\begin{definition} \label{def:convex}
	Given a decorated pointed disc $(S, \mathbf{\epsilon}, \mathbf{w}, \gamma )$, partial convex data $(H, J)$ and a family of partial convex data $\mathbf{H}$ with compact support relative to $(H,J)$ which extends to $\bar{E}$, let ${I_\mathbf{H}}$ be the
	associated almost complex structure on $\bar{E}$. A function $G : E \to
	\R$ will be called ${I_\mathbf{H}}$-convex if there exists a $C_G > 0$
	such that for every $C \in \R$,
	\begin{enumerate}[label=(\roman*), ref=\thetheorem(\roman*)]
		\item \label{def:convex:1} $G |_{\bar{E}_{C}}$ is proper,
		\item \label{def:convex:2} for $G > C_G$,
		\begin{equation} \omega_G (v, {I_\mathbf{H}}v) := \iota_{I_{\mathbf{H}}v} \iota_v ( - \tn (\tn G \circ {I_\mathbf{H}}
		)) \geq 0 .
		\end{equation}
	\end{enumerate}
	The constant $C_G$ will be called a sublevel bound for $G$.
\end{definition}

Given a function $g : M \to \mathbb{R}$ and a surface $S$, we say that $G : M \times  S \to \mathbb{R}$ is \textit{decomposable by} $g$ \textit{with bound} $B$ if $G = f \circ \pi_S + g \circ \pi_M$ for some subharmonic function $f :S \to \mathbb{R}$ and $\max \{f|_{S^1} \} = B$. 

\begin{definition}
	A function $g : M \to \mathbb{R}$ suppresses $H$ with bound $B$ if, for every $d \in \mathbb{N}$, $S \in \mathcal{R}^{d + 1}$ and every family $\mathbf{H} \in \mathcal{H}_{S, H, J}$ in a neighborhood of the constant family $\{(H, J)\}$, there exists a function $G : M \times S \to \mathbb{R}$, decomposable by $g$ with bound $B$, which is	$I_\mathbf{H}$-convex. We will say that $H$ is \textit{partially convex} with bound $B$ if it has a suppressing function $g$.
\end{definition}
If the bound $B$ is not relevant, we will simply call $H$ \textit{partially convex}. The reason we wish to find suppressing functions is twofold. First and foremost, these functions give a
maximum principle which constrains holomorphic discs to a compact region in $M$, thereby ensuring good compactifications
of their moduli spaces. Second, for large enough values, the sublevel sets of $I_\mathbf{H}$-convex functions
satisfy this same maximum principle. Coupling this to an argument on the flow of the Hamiltonian, we obtain functors
from partially wrapped categories of sublevel sets to that of the entire manifold $M$. To be precise about the type of maximum principle we obtain, we observe the following Lemma.

\begin{lemma} \label{lem:finalest} Suppose $g: M \to \mathbb{R}$ suppresses $H$ with bound $B$ and $u : S \to M$ satisfies Floer's equation \ref{eq:Floer} for a family of partial convex data $\mathbf{H}$ sufficiently close to the constant family, then \[ \| g \circ u \|_{L_\infty (S)} \leq \|(g \circ u)|_{\partial S} \|_{L_\infty (\partial S)} + 2B . \]
\end{lemma}
\begin{proof}
	Let $G : M \times S \to \mathbb{R}$ be an $I_{\mathbf{H}}$-convex function decomposable by $g$ with bound $B$ so that $G = f \circ \pi_S + g \circ \pi_M$. Note that since $f$ is subharmonic, it satisfies the maximum principle and $\|f\|_{L_\infty (S)} \leq B$ for all $p \in S$. Since $u$ satisfies equation \eqref{eq:Floer}, the section $s_u : S \to M \times S$ given by $s_u (p) = (u(p), p)$ is $I_{\mathbf{H}}$-holomorphic. Thus, by \cite[Lemma~9.2.9]{ms}, the composition $G \circ s_u$ is subharmonic. The maximum principle then gives 
	\begin{align*}\| g \circ u \|_{L_\infty (S)} & =  \| G \circ s_u - f   \|_{L_\infty (S)}, \\ & \leq \|G \circ s_u \|_{L_\infty (S)} + \|f\|_{L_\infty (S)}, \\ & \leq \|(G \circ s_u)|_{\partial S}\|_{L_\infty (\partial S)} + B,\\ & = \| (g \circ u)|_{\partial S} + f|_{\partial S} \|_{L_\infty (\partial S)} + B, \\ & \leq \| (g \circ u)|_{\partial S}\|_{L_\infty (\partial S)} + 2B. 
	\end{align*}
\end{proof}
The next technical lemma gives a criterion for showing that a given Hamiltonian is partially convex.

\begin{lemma}\label{lemma:Hconvex} Suppose $g : M \to \mathbb{R}$, and $J \in \mathcal{J} (M)$. If 
	\begin{enumerate}[label=(\roman*), ref=\thetheorem(\roman*)]
		\item \label{lemma:Hconvex:1} $g$ is proper on sublevel sets of $H$,  
		\item \label{lemma:Hconvex:2} $g$ is $J$-convex,
		\item \label{lemma:Hconvex:3} there exists $C_0 \geq 0$ such that \[ \max\{\mathcal{L}_{X_H} (g),  \mathcal{L}_{JX_H} ( g ) \}  \leq C_0 \]
		\item \label{lemma:Hconvex:4} there exist $C_1 > 0$ such that \[ \max \left\{ \|X_H\|_g , \| \iota_{X_H} \eta_g \|_g,  \|\mathcal{L}_{X_H} \tn g \|_g , \| \mathcal{L}_{X_H} \tn^c g\|_g \right\} \leq C_1  ,\]
	\end{enumerate}
	then  $g$ suppresses $H$ with bound \[
	C_1^2 L + C_0 L, \] where $L$ is the constant, independent of $g$, $H$ and $J$, from \eqref{eq:gamma4}.
\end{lemma}
In practice, we will be working with integrable complex structures. In these cases, verifying the bound involving $\| \iota_{X_H} \eta_g \|_g = 0$ is unnecessary. 
\begin{proof}
	Let $(S, \mathbf{\epsilon}, \mathbf{w}, \gamma )$ be a decorated pointed disc and $G =f \circ \pi_S + g \circ \pi_M$
	for some function $f : S \to \mathbb{R}$. Let $\mathbf{H}$ be the trivial family of partial convex data relative to $H$ and $I_\mathbf{H}$ the associated almost complex structure.  We compute  
	\begin{align*} 
	\begin{split}
	\omega_G  & =  -\tn(\tn G \circ I_\mathbf{H}) \\ & = -\tn( \tn g \circ J - (\tn g (J X_H)) \gamma + (\tn
	g (X_H)) \gamma \circ j + \tn f \circ j ), \\
	& =  -\tn( \tn g \circ J) + (\tn g (J X_H)) \tn \gamma + \tn (\tn g (J X_H)) \wedge \gamma - \tn (\tn g (X_H)) \wedge
	(\gamma \circ j) \\ &  \hspace{1cm}
	- \tn g (X_H)  \tn (\gamma \circ j) - \tn (\tn f \circ j),  \\ 
	& =  \omega_g^M  + \tn (\tn g (J X_H)) \wedge \gamma - \tn (\tn g (X_H))+ \tn g (J X_H ) \tn \gamma \wedge
	(\gamma \circ j) \\ &  \hspace{1cm}  - \tn g (X_H) \tn (\gamma \circ j) + \omega_f^S.
	\end{split}
	\end{align*}
	Here we write $\omega_g^M$ and $\omega_f^S$ for $\pi_M^* \omega_g$ and $\pi_S^* \omega_f$ respectively. To shorten the notation, we also write $\tau_1 = \tn (\tn g (JX_H))$ and $\tau_2 = \tn (\tn g (X_H))$ to obtain
	\begin{align} \label{eq:longform} \omega_G = \omega_g^M  + \tau_1 \wedge \gamma - \tau_2 \wedge
	(\gamma \circ j) + \tn g (J X_H ) \tn \gamma  - \tn g (X_H)  \tn^c \gamma +  \omega_f^S. \end{align}
	Basic Cartan calculus gives us $\tau_1 = \mathcal{L}_{X_H} \tn^c g + \iota_{X_H} \omega_g$ and $\tau_2 = - \mathcal{L}_{X_H} \tn g$.

	Let $v \in T_{m} M $ and $w \in T_p S$. Applying property \ref{lemma:Hconvex:4} gives 
	\begin{align} \label{eq:t2b}  | \tau_2 (v) | \leq C_1\|v \|_g . \end{align} Using equation \eqref{eq:antisymmetric},  $\omega_g (v_1,J v_2) = \left< v_1, v_2\right>_g + \eta_J (v_1,v_2)$, and applying Cauchy-Schwarz along with property \ref{lemma:Hconvex:4} gives
	\begin{align} \label{eq:omxjx} |\omega_g (X_H , v)| & \leq 2 C_1 \|v\|_g, & |\omega_g (J X_H , v)| & \leq 2 C_1 \|v\|_g. \end{align}  Thus we also have
	\begin{align} \label{eq:t1b} | \tau_1 (v) | \leq 3 C_1 \| v \|_g .\end{align}
	
	Now, for $x,y \in \mathbb{R}$ let $u = (xv,yw) \in T_{(m,p)} (M \times S)$ and define the quadratic form 
	\[\omega_G (u, I_{\mathbf{H}} u) = q (x,y) = a x^2 + 2 b x y + c y^2 . \]
	Using equations \eqref{eq:gromovtrick} and \eqref{eq:longform},  one observes that $a = \omega_g^M ( v , J v)$  which is positive by property \ref{lemma:Hconvex:2} and equals $\| v\|^2_g$. One also computes the cross term $2 b$ to be
	\begin{align*} 2b & = - \gamma (w) \omega_g (v, J X_H ) + \gamma (jw) \omega_g (v, X_H) + \tau_1 (v) \gamma ( jw ) , \\ & \hspace{1cm}  - \tau_1 (Jv) \gamma (w) + \tau_2 (v) \gamma (w) + \tau_2 (Jv) \gamma (jw).
	\end{align*}
	Using the inequalities \eqref{eq:t1b}, \eqref{eq:omxjx}, \eqref{eq:t2b}, and taking the standard metric $\| \|_S$ on the disc $S$ induced by $j$ and $\omega^S$, we then have
	\begin{align} \label{eq:Hconvexineq1} |b| \leq 3 C_1 \|v\|_g \left( |\gamma (w)| + | \gamma (jw )| \right) \leq 3 C_1 \|v \|_g \| \gamma \|_S \|w\|_S . \end{align}
	Again, using equations \eqref{eq:gromovtrick} and \eqref{eq:longform}, one calculates the $y^2$ coefficient of $q$ to be 
	\begin{align*} c & = \gamma (w)^2 \tau_1 (J X_H) - \gamma (w) \gamma (jw) ( \tau_1 (X_H ) + \tau_2 (JX_H) ) + \gamma (jw)^2 \tau_2 (X_H ) , \\ & \hspace{1cm} + \mathcal{L}_{J X_H} (g) \tn \gamma (w, jw) - \mathcal{L}_{X_H} (g) \tn^c \gamma (w , jw) + \omega_f^S (w, jw) . \end{align*}
	Using \ref{lemma:Hconvex:4} twice we see that, for any integer $l$, $|\tau_1 (J^l X_H)| \leq 3 C_1^2$ and $|\tau_2 (J^l X_H) | \leq C_1^2$. Combining this observation with \ref{lemma:Hconvex:3} we obtain
	\begin{align} \label{eq:Hconvexineq2}
	|c - \omega_f (w, jw) | & \leq \left( 8 C_1^2  \| \gamma \|_S^2 + C_0 \| \tn \gamma \|_S + C_0 \| \tn^c \gamma\|_S \right) \|w\|_S^2.
	\end{align}
	Combining inequalities \eqref{eq:Hconvexineq1} and \eqref{eq:Hconvexineq2} we obtain
	\begin{align} \label{eq:Hconvexineq3}
	\begin{split}
	\left| \frac{ b^2  - a (c - \omega_f (w, jw))}{a \|w\|_S^2} \right| & \leq \frac{b^2}{a\|w\|_S^2} + \frac{|c - \omega_f (w, jw) |}{\|w\|_S^2}, \\ & \leq  17 C_1^2 \|\gamma \|_S^2 + C_0 ( \| \tn \gamma \|_S + \|\tn^c \gamma \|_S ). 
	\end{split} 
	\end{align}
	Now, since $a$ is positive, $q$ is positive definite if and only if $c > b^2 / a$ or $\omega_f (w, jw) > b^2 / a - (\omega_f(w, jw) - c)$. In local coordinates $(s,t)$ on $S$, let $f = B (s^2 + t^2)$, reducing this inequality to $B \|w\|^2  > b^2 / a - (\omega_f (w, jw) - c)$. 
	Thus for $B > 17 C_1^2 \|\gamma \|_S^2 + C_0 ( \| \tn \gamma \|_S + \|\tn^c \gamma \|_S )$, equation \eqref{eq:Hconvexineq3} implies $q$ is positive definite and, as $u = (v,w)$ was arbitrary, $g$ is $I_{\mathbf{H}}$ convex. As convexity is an open condition, there exists a neighborhood around the trivial perturbation data $\mathbf{H}$ for which $g$ will remain convex.
\end{proof}

Lemma~\ref{lemma:Hconvex} gives one a practical test to verify that the moduli spaces of discs $\mathcal{R}^{d + 1, \mathbf{w}} (\mathbf{x})$, which define the structure maps in the Fukaya category, admit the  compactification necessary for the $A_\infty$-structure.  To state this precisely, we prove the last lemma of this section.

\begin{lemma} \label{lemma:compactness} Suppose $H$ is partially convex, then there exists a compact
	subspace $K \subset M$ such that any $(S, u) \in \mathcal{R}^{d + 1,
		\mathbf{w}} (\mathbf{x} )$ satisfies  $u (S) \subset K$.
\end{lemma}
\begin{proof}
	We cite Lemma 7.2 of \cite{as} which implies there exists $C$ for which $H (u (S)) < C$ for all $(S, u) \in \mathcal{R}^{d + 1, \mathbf{w}} (\mathbf{x} )$. In our notation, this implies $u(S) \subset \bar{E}_C$ for every $(S,u) \in \mathcal{R}^{d + 1, \mathbf{w}} (\mathbf{x})$.	Let $(S, u) \in \mathcal{R}^{d + 1, \mathbf{w}} (\mathbf{x})$ be an arbitrary stable popsicle map and write $\sigma : S \to E$ for the associated $I_{\mathbf{H}}$-holomorphic section. Let $K_1 \subset \bar{E}_{C}$ be a compact subset containing $L_k \times S \cap \bar{E}_{C}$ for 	$0 \leq k \leq d$ and $\sigma (\epsilon^{k} (s, -))$ with $s \gg 0$ for $1 \leq k \leq d$ and $s \ll 0$ for $k = 0$. Such a set exists by the assumption \eqref{eq:lagpr2} and property \eqref{eq:disc2}. Let $G : E \to \mathbb{R}_{\geq 0}$ be an $I_\mathbf{H}$-convex function. Since $G|_{\bar{E}_C}$ is proper, we may assume that there is a constant $A$ for which $G(K_1) \subset [0, A)$. Then there exists a neighborhood $U$ of $\partial (S \cup \{\zeta^0, \ldots, \zeta^d\} )$ with $G (\sigma (U_{\varepsilon} )) \subset [0,A)$. Taking $r$ sufficiently close to $1$ and $S_r$ the radius $r$ sub-disc of $S$ this implies that $G (\sigma (\partial S_{r} )) \subset [0, A)$. But, since $\sigma$ is $I_{\mathbf{H}}$-holomorphic, we can compute the Laplacian $\Delta (G \circ \sigma ) = \omega_G (\partial_s u, I_{\mathbf{H}} \partial_s u) \geq 0 $ (see \cite[Section 9.2]{ms}). Thus $G \circ \sigma$ is subharmonic and obeys the maximum principle showing that $G (\sigma (S)) \subset [0, A)$ or $\sigma (S) \subset G|_{\bar{E}_C}^{-1}([0, A])$. As  $G$ is proper on $\bar{E}_C$, the sublevel set $K_2 := G|_{\bar{E}_C}^{-1}([0, A])$ is compact. Letting $K \subset M$ be the compact image $\pi_M (K_2)$ we see $u(S) = (\pi_M \circ \sigma) (S) \subset K$ for all $(S,u)\in \mathcal{R}^{d + 1, \mathbf{w}}(\mathbf{x})$.
\end{proof}

This result then yields the following theorem as a corollary.
\begin{theorem} \label{thm:exist}
	If $H : M \to \mathbb{R}$ is partially convex, then $\wrapped{M}{H}$ is a well defined $A_\infty$-category.
\end{theorem}

\section{\label{sec:local} Local models} 
Examples of partially wrapped categories have been explored in complex dimension $1$ \cite{hkk, lee}. 
In this section we will extend these categories to higher dimensions, pursuing a differing approach than \cite{sylvan}, so long as our manifold and wrapping Hamiltonian meet certain basic and somewhat stringent requirements. While the local structure investigated here can be generalized as in \cite{nadler16},  it is sufficiently rich to describe mirror categories to quasi-affine toric stacks. 

\subsection{Simplicial Lagrangians} Let  $\omega$ be the canonical symplectic structure on $T^* \mathbb{R}^m$. Denote the set $\{1, \ldots, m\}$ by $[m]$ and $\mathcal{P} ([m])$ its power set. Suppose $K \subseteq \mathcal{P} ([m])$ is an abstract simplicial complex with vertex set $[m]$. We will write $\sigma \in K$ for membership of simplices as well as $j \in K$ when $j \in [m]$ and $\{j\} \in K$.
For any simplex $\sigma \in K$, we define a conical Lagrangian in $T^* \mathbb{R}^n$ as 
\begin{align} \label{eq:stratadef}
L_\sigma = \{ (x, y) \in T^*_x \mathbb{R}^m : x_i = 0 , y_i \geq 0 \text{ for } i \in \sigma , y_j = 0 \text{ for } j
\not\in \sigma \}.
\end{align}
For example, $L_\emptyset \cong \mathbb{R}^m$ is the zero section of $T^* \mathbb{R}^m$. We write $L_\sigma^\circ$ for the relative interior of $L_\sigma$. Taking the union of these Lagrangians gives the singular,
conical Lagrangian
\begin{align} \label{eq:singlagdef}
L_K := \bigcup_{\sigma \in K} L_\sigma .
\end{align}

Equating $\mathbb{C}^m = T^* \mathbb{R}^m$ gives the former the usual Hermitian metric. For any subset $S \subseteq T^* \mathbb{R}^m$, let $d_S : T^* \mathbb{R}^m \to \mathbb{R}$ be the distance from $S$ relative to this
metric. Define the kinetic energy $\tilde{H}_K : T^* \mathbb{R}^m \to \mathbb{R}$ relative to $L_K$ to be 
\begin{align}
\tilde{H}_K  = \frac{1}{2} d_{L_K}^2  
\end{align}
and for any $\varepsilon > 0$, denote the sublevel set by
\begin{align} \tilde{U}_K (\varepsilon )= \left\{ p \in T^* \mathbb{R}^m : \tilde{H}_K (p)  \leq \varepsilon \right\} \end{align}
For $\sigma \in K$, let 
\begin{align} \tilde{U}_\sigma (\varepsilon ) = \left\{ p \in \tilde{U}_K (\varepsilon ) : \tilde{H}_K (p) =  \frac{1}{2} d_{L_\sigma}^2 (p) \right\} 
\end{align}
be the closed subset of points closest to the strata $L_\sigma \subset L_K$. We note that the Hamiltonian vector field of $\tilde{H}_K$ is defined in $\tilde{U}_\sigma (\varepsilon )$ and easily calculated to be  
\begin{align} \label{eq:hamvf}
\tilde{X}_{K} & = \sum_{j \not\in \sigma} y_j \partial_{x_j} - \sum_{j \in \sigma} x_j \partial_{y_j}. 
\end{align}

When $K = \emptyset$, this is simply the kinetic energy which induces cogeodesic flow. In this case, one can
take this as a local model for the wrapped Fukaya category on the cotangent bundle of a manifold. However, for any non-trivial $K$, the function $\tilde{H}_K$ is not smooth near points equidistant to different strata. We take 
\begin{align}
\sing{K} = \{ p \in T^* \mathbb{R}^m : \tilde{H}_K \text{ is not differentiable at } p\}. 
\end{align} 
We will write $\tilde{X}_K$ for the discontinuous Hamiltonian vector field of $\tilde{H}_K$ defined off of $\sing{K}$. Note that $\sing{K}$ is contained in the set of elements of $T^* \mathbb{R}^m$ which are equidistant to different strata $L_\sigma$ of $L_K$. For any $J \subset K$, we introduce the notation 
\begin{align}
\singstr{K}{J} = \{p \in \sing{K} : 1/2 d_{L_\sigma}^2 (p) = \tilde{H}_K (p) \text{ for all } \sigma \in J\}.
\end{align}

\begin{figure}[t]
	\begin{picture}(0,0)%
	\includegraphics[scale=.95]{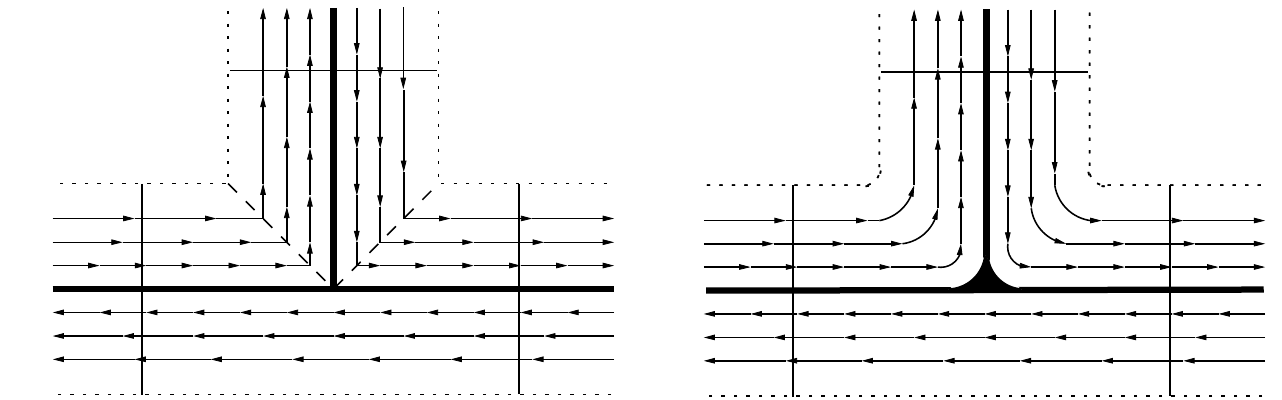}%
	\end{picture}%
	\setlength{\unitlength}{3947sp}%
	\begin{picture}(6135,1901)(3941,-8343)
	\put(4504,-7253){\makebox(0,0)[lb]{\smash{$L_0$}}}
	\put(4741,-6870){\makebox(0,0)[lb]{\smash{$L_1$}}}
	\put(6213,-7247){\makebox(0,0)[lb]{\smash{$L_2$}}}
	\put(3856,-7860){\makebox(0,0)[lb]{\smash{$L_K$}}}
	\put(9781,-7854){\makebox(0,0)[lb]{\smash{$Z_{\delta, K}$}}}
	\put(7469,-7260){\makebox(0,0)[lb]{\smash{$L_0$}}}
	\put(7696,-6870){\makebox(0,0)[lb]{\smash{$L_1$}}}
	\put(9218,-7254){\makebox(0,0)[lb]{\smash{$L_2$}}}
	\end{picture}
	\caption{\label{fig:A2}Generating Lagrangians, $\tilde{H}_K$ for $K$ equal to a point and an $A_2$ Hamiltonian $H$.}
\end{figure}
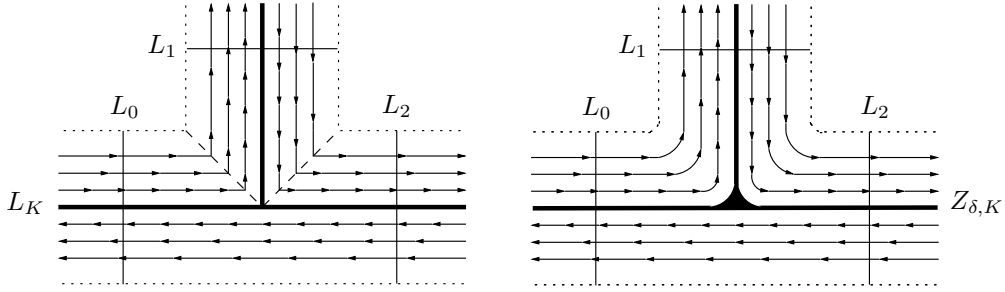

\begin{example} \label{example:a2} A basic example
	which will play a dominant role is the one dimensional case $n = 1$ with $K = \{\emptyset , \{1\}\}$. 
	The Lagrangian $L_K$, its neighborhood $\tilde{U}_K (\varepsilon)$ and the Hamiltonian vector field $\tilde{X}_K$ for $\tilde{H}_K$ for this
	example are illustrated in Figure \ref{fig:A2}. 
\end{example}

To obtain a wrapped Fukaya category using the Hamiltonian
$\tilde{H}_K$, one would first search for a smooth plurisubharmonic approximation $H_K$. Unfortunately, an arbitrarily close $C^0$-approximation does not exist in general as the function $\tilde{H}_K$ is not plurisubharmonic as a continuous function. This can be seen through direct computation in the case of Example \ref{example:a2}. Nevertheless, for $K = \{\emptyset, \{1\}\}$ one may perturb $\tilde{H}_K$ by choosing $0 < \delta < \varepsilon$ and considering the set 
\begin{align} \label{eq:A2domain}
Z_{\delta} := L_K \cup \left\{ (x,y) \in T^* \mathbb{R} :0 \leq  y \leq \delta , |x| \leq \delta , \|(x,y) \pm (\delta , \delta ) \| \geq \delta \right\}.
\end{align}
The kinetic energy relative to $Z_{\delta}$
\begin{align} \label{eq:A2Ham}
\tilde{H}_{\delta} = \frac{1}{2} d^2_{Z_{\delta}}
\end{align}  is a differentiable plurisubharmonic function.
\begin{definition}
	Let $K$ be the simplicial set of a point. An $i$-convex smooth Hamiltonian $H : \tilde{U}_K (\varepsilon ) \to \mathbb{R}$ will be called an $A_2$-Hamiltonian if there exists $0 < \delta < \varepsilon$ and an isotopy $\{X_t\}_{t \in [0,1]}$ of Hamiltonian vector fields with $X_1 = X_{H}$, $X_0 = X_{\tilde{H}_{\delta}}$ and $\tn \tilde{H}_{\delta} (X_t (p)) = 0$ for all $t$ and $\tilde{H}_{\delta} (p) \geq \delta / 2$.
\end{definition}
This definition ensures that an $A_2$-Hamiltonian will look the same as $X_{\tilde{H}_{\delta}}$ outside of a neighborhood of $\sing{K}$.  For general simplicial complexes $K$, instead of focusing on a particular smoothing of $\tilde{H}_K$, we will consider Hamiltonians which satisfy a set of axioms which are properties of $\tilde{H}_K$. In order to motivate and state these axioms, we first must establish some notation. 

Given a simplex $\sigma \in K$, recall that the star of $\sigma$ consists of the simplices $\str (\sigma) = \{\tau \in K : \sigma \subseteq \tau\}$. For any subset $T \subseteq K$, its closure  is the union of the  faces of the simplices in $T$, or $\clos (T) = \{\tau : \tau \text{ is a face of } \tau^\prime \in T\}$.  The link of $\sigma \in K$ is the simplicial subset $\lk (\sigma ) = \clos (\str (\sigma)) - \str (\clos (\sigma ))$.  Let $\sigma \in K$ and $I$ be a set of points in $\lk (\sigma )$. Write the simplicial subset of $\lk (\sigma)$ obtained by restriction to $I$ as $\lk_I (\sigma )$. 

We identify $T^* \mathbb{R}^m$ with $\mathbb{C}^m$ and, for any $I \subset [m]$, write $\mathbb{C}^I$ for the vector space with standard basis $\{e_i : i \in I\}$. Let $\pi_I : \mathbb{C}^m \to \mathbb{C}^{I}$ be the projection and $\iota_I : \mathbb{C}^{I} \to \mathbb{C}^m$ the inclusion. 
\begin{definition} 
	Suppose $\sigma \in K$ and $I \subseteq [m]$ a subset of points in $\lk (\sigma )$. The vector 
	\begin{align} \label{eq:uIdef}
	u_I = \sum_{j \not\in I \cup \sigma} a_j e_j + \sum_{j \in \sigma } i b_j e_j
	\end{align}
	will be called \textit{sufficiently displaced} if $a_j, b_j \in \mathbb{R}$,  $|a_j| > \sqrt{2 \varepsilon}$ and $b_j > \sqrt{2 \varepsilon}$ for all $j \not\in I$. 
\end{definition}
Consider the Lagrangian subspace
\begin{align} \label{eq:sigbk} L_{\sigma, [m] \backslash I} = \left\{\sum_{j \not\in I} w_j e_j : \Re (w_j ) = 0, \textnormal{ for } j \in \sigma, \Im (w_j ) = 0, \textnormal{ for } j \not\in \sigma \right\}  \end{align} 
of $\mathbb{C}^{[m] \backslash I}$.
Note that if $u_I$ is sufficiently displaced, then there exists a neighborhood $U$ of the origin in $\mathbb{C}^{[m]\backslash I}$ such that $u_I + u$ is sufficiently displaced for all $u \in U \cap L_{\sigma, [m] \backslash I}$. More precisely, consider the affine subspace 
\[ V_{u_I} = \left\{ z \in \mathbb{C}^m : \pi_{[m] \backslash I} (z) = u_I \right\} \]
and take $L_{u_I} = L_K \cap V_{u_I}$ and $V_{u_I}  (\varepsilon ) = V_{u_I } \cap \tilde{U}_\varepsilon (K)$. Then we have the following lemma.
\begin{lemma} \label{lemma:prop1}
	For $\sigma \in K$, $I$ a set of points in $\lk (\sigma )$, $u_I$ sufficiently displaced, the affine inclusion $\iota_{I} + {u_I} : (\tilde{U}_{K|_I} (\varepsilon ), L_{K|_I}) \to (V_{u_I} (\varepsilon  ), L_{u_I} )$ is an isomorphism of pairs. Furthermore, there is a neighborhood $U$ of the origin in $\mathbb{C}^{[m] \backslash I}$ and a real vector field $Y$ in $U \subset \mathbb{C}^{[m] \backslash I}$ which vanishes on $L_{\sigma, [m] \backslash I} \cap U$ for which 
	\begin{align*} \tilde{X}_K = (\iota + u_I)_* \left(\tilde{X}_{K|_I} \right) + (\iota_{[m] \backslash I})_* (Y )
	\end{align*}
	for all elements of $\left( \tilde{U}_{I} (\varepsilon) \oplus U \right) + u_I \subset \mathbb{C}^m$.
\end{lemma}
\begin{proof}
	Let $\phi$ be the affine inclusion $\iota_I + u_I$.  Since $I$ is a subset of points in $\lk (\sigma)$, we have $\tau \in K|_I$ if and only if  $\tau \sqcup \sigma \in K$. Using equation~\eqref{eq:stratadef}, and the definition of $u_I$ in equation~\eqref{eq:uIdef}, one then has $\phi (L_\tau ) = L_{\sigma \sqcup \tau}  \cap \textnormal{im} (\phi )$. This establishes that that $\phi$ is a holomorphic isomorphism between the pairs $(U_{K|_I} (\varepsilon ), L_{K|_I})$ and $(V_{u_I} (\varepsilon  ), L_{u_I} )$. Moreover, in a sufficiently small neighborhood $U$ of the origin in $\mathbb{C}^{[m] \backslash I}$, we have 
	\begin{align} \label{eq:decomp}
	L_{K} \cap \left[(\mathbb{C}^I \oplus U ) + u_I\right] = \phi ( L_{K|_I} ) \oplus \left(L_{\sigma, [m] \backslash I} \cap U \right).
	\end{align}
	
	This implies that $\tilde{H}_K = \tilde{H}_{K|_I} + \frac{1}{2} d_{L_{\sigma, [m] \backslash I}}^2$ in $(\mathbb{C}^I \oplus U ) + u_I$. As the standard symplectic structure in $\mathbb{C}^m$ splits with respect to this direct sum, the Hamiltonian vector field $\tilde{X}_K$ also splits as the sum $(\iota_I + u_I)_* (X_{\tilde{H}_{K|_I}}) + (\iota_{[m] \backslash I})_* (Y)$ where $Y$ is the Hamiltonian vector field of $\frac{1}{2} d_{L_{\sigma, [m] \backslash I}}^2$.
\end{proof}
We now turn the property in Lemma~\ref{lemma:prop1} into a definition.
\begin{definition} \label{def:decomposable}
	Let $K$ be a simplicial complex. A collection of smooth Hamiltonians $\{H_I : \tilde{U}_{K|_I} (\varepsilon ) \to \mathbb{R} \}_{I \subset K}$ will be called decomposable if, for every $\sigma \in K$, $I$ a set of points in $\lk (\sigma )$ and $u_I$ sufficiently displaced, there exists a vector field $Y$ on a neighborhood $U \subset \mathbb{C}^{[m] \backslash I}$, vanishing along $L_{\sigma, I}$ for which 
	\begin{align} \label{eq:splitvf}
	X_{H} = (\iota_{I} + u_I)_* ( X_{H_I})  + (\iota_{[m] \backslash I})_* ( Y)
	\end{align}
	on $U$.
\end{definition}

Another important property that the discontinuous Hamiltonian $\tilde{H}_K$ satisfies is stated in the following lemma.
\begin{lemma} \label{lemma:prop2}
	For any $i \in [m]$,  $a \in \mathbb{R}$, if $f = (x_i - a)^2$ or $f = (y_i- a)^2$, then $\tilde{X}_K^2 (f) \geq 0$ when it is defined.
\end{lemma}
\begin{proof}
	This follows from an elementary application of equation~\eqref{eq:hamvf} for $I = [m]$ and any $\tau \in K$.
\end{proof}
Again, we convert this conclusion into a definition, albeit in a slightly weakened form. First, consider a Hamiltonian $H : U \to \mathbb{R}_{\geq 0}$ on an open subset of $L_K \subset T^* \mathbb{R}^m$ and $Z = H^{-1} (0)$ to be the zero level set. Assume that there is a retraction $\pi : U \to Z$. 
\begin{definition} \label{def:expanding}
	Given a Hamiltonian $H : U \to \mathbb{R}_{\geq 0}$, let  $\varphi_t: U \to U$ be its time $t$ flow. Let $p \in U$ be such that $\pi (p) \in L^{sm}_\sigma$ for some $\sigma \in U$. We say that  $H$ is expanding if, given  $f = (x_i - a)^2$ for $i \not\in \sigma$ or $(y_i - a)^2$ for any $i \in \sigma$ and $a \in \mathbb{R}$, if $f(p) = 0$ then either $f (\pi (\phi_t (p)))$ is monotonic or $f (\pi (\phi_t (p))) \ne 0$ for $t > 0$.
\end{definition}
One final property held by the Hamiltonians $H_K$ that we wish to generalize is their behavior under taking simplicial cones. Given a simplicial complex $\bar{K}$ on $[m-1]$, the cone $\text{cone} (\bar{K})$ of $\bar{K}$ is the simplicial complex $\bar{K} \cup \{\sigma \cup \{m\} : \sigma \in \bar{K}\}$ on $[m]$. Note that $L_{\text{cone} (\bar{K})} = L_{\bar{K}} \times L_{m}$ where $m$ denotes the simplicial complex of the point $m$ as in Example~\ref{example:a2}.
\begin{definition} \label{def:stable} A collection of Hamiltonians $\{H_I : \tilde{U}_{K|_I} (\varepsilon ) \to \mathbb{R} \}$ is stable if, whenever $K|_I = \textnormal{cone} (K|_{I \backslash \{i\}})$, $H_I = H_{I \backslash \{i\}} + H_{\{i\}}$.
\end{definition}

We wish to preserve Lemmas \ref{lemma:prop1}, \ref{lemma:prop2} and stability when we consider a smooth approximation of $\tilde{H}_K$. We also ask that our approximation equals $\tilde{H}_K$ near any strata $L_\sigma \subset L_K$ and away from the singular locus. Define a neighborhood 
\begin{align}
\singe_K = \bigcup_{p \in \sing{K}} B (p, \varepsilon)  \subset T^* \mathbb{R}^m
\end{align}
around the singular set of $\tilde{H}_K$. Let $\{U_{K|_I} (\varepsilon) \subset T^* \mathbb{R}^I\}_{I \subseteq [m]}$ be a collection of neighborhoods of $L_{K|_I} \subset T^* \mathbb{R}^I$. We assume that $\tilde{U}_{I} (\varepsilon / 2) \subset U_{I}(\varepsilon) \subset \tilde{U}_I (2\varepsilon)$ for all $I$ and that they are compatible in the  sense that for any $\sigma \in K$, $I \subset \text{lk} (\sigma )$ and sufficiently displaced $u_I$, $\iota_I + u_I : (U_{K|_I} (\varepsilon), L_{K_I}) \to (V_{u_I} \cap U_K , L_{u_I})$ is an isomorphism of pairs. This implies that  there exists a neighborhood $U \subset \mathbb{C}^{[m] \backslash I}$ of the origin such that $(U_I (\varepsilon) + u_I) \times U = \{z \in U_{K} : \pi_{[m] \backslash I} \in u_I + U\}$.

\begin{definition} \label{def:ham}
	Let $K$ be a simplicial complex. A stable, decomposable, collection of smooth expanding Hamiltonians $\{H_I : U_{K|_I} (\varepsilon ) \to \mathbb{R}\}_{I \subset [m]}$ will be called a compatible smoothing of $\tilde{H}_K$ if:
	\begin{enumerate}[label=(\roman*), ref=\thetheorem(\roman*)]
		\item \label{def:ham:0} $H_{[m]}$ is $i$-convex and partially convex,
		\item \label{def:ham:1} for every point $I = \{i\} \subseteq [m]$, $H_I$ is an $A_2$-Hamiltonian,
		\item \label{def:ham:2} for every $I \subset [m]$ and $\sigma \in K|_I$, $H_I|_{U_\sigma (\varepsilon ) \backslash \singe_{K|_I} } = \tilde{H}_{K|_I}$.
	\end{enumerate}
\end{definition}
Note that if $\{ H_I \}_{I \subset [m]}$ is a compatible smoothing for $K$ and $J \subseteq [m]$ then $\{H_{I} \}_{I \subset J}$ is a compatible smoothing of $K|_J$. We now establish the existence of a compatible smoothing of $\tilde{H}_K$. The existence of an $A_2$-Hamiltonian is shown in equation~\eqref{eq:A2Ham} as kinetic energy relative to the space $Z_{\delta}$ defined in equation~\eqref{eq:A2domain}. For $K_{m}$ equal to the $(m - 1)$-simplex, we may take $H_{\delta, m} : T^* \mathbb{R}^{m} \to \mathbb{R}$ to be kinetic energy relative to $Z_\delta^{m}$. Then in a neighborhood of $L_{K_m} \subset Z^{n}_\delta$, $H_{\delta, m}$ is smooth and $i$-convex (as it is the sum of $H_\delta$ with respect to the $m$-coordinates). 

\begin{figure}[b]
	\begin{picture}(0,0)%
	\includegraphics{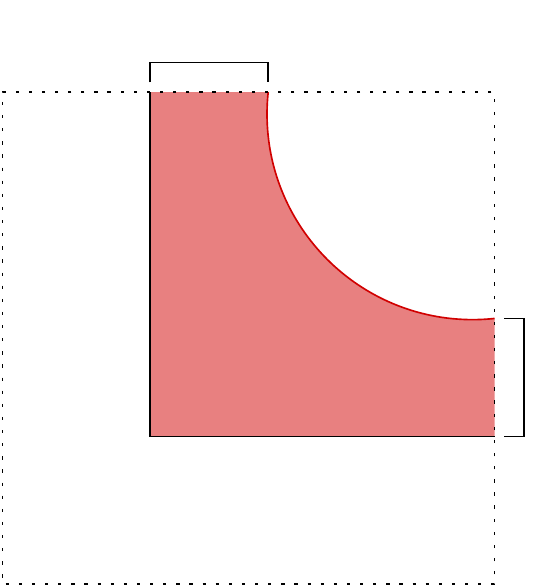}%
	\end{picture}%
	\setlength{\unitlength}{4144sp}%
	\begin{picture}(2502,2679)(7414,-3673)
	\put(8311,-1141){\makebox(0,0)[lb]{\smash{$\alpha$}}}
	\put(9901,-2786){\makebox(0,0)[lb]{\smash{$\alpha$}}}
	\end{picture}%
	\caption{\label{fig:LJ} The lowered strata $L_{I}^{< \alpha}$ for $I = [2]$.}
\end{figure}
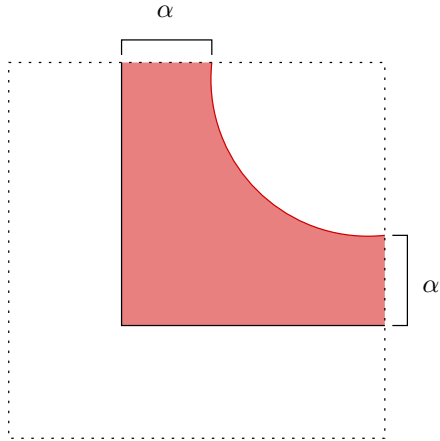
For any $I \subseteq [m]$, write $e_I = \sum_{i \in I} e_i \in \mathbb{C}^m = T^* \mathbb{R}^m$, where $\{e_i\}$ is the standard basis. For any $\alpha > 0$ and subset $X \subset \mathbb{C}^m$, take $B_\alpha (X) = \{ z : d_X (z) < \alpha \}$ to be the open radius $\alpha$-neighborhood of $X$. Consider the subsets
\begin{align}
L_I^{> \alpha}  & = L_I \cap B_\alpha (L_I + 2 \alpha i e_I), \\
L_I^{\leq \alpha} & = L_I \backslash L_I^{> \alpha}.
\end{align}
Note that, in addition to the corners of $\partial L_I$, the subspace $L_I^{ \leq \alpha}$ acquires a $(m -1)$-dimensional isotropic smooth boundary component (except in the case of $I = \emptyset$). We write this boundary component as $\partial_+ L_I^{\leq \alpha}$.

Consider the $\varepsilon$-neighborhoods of $B_\varepsilon (L_I^{> \alpha})$ and $B_\varepsilon (L_I^{\leq \alpha})$. Observe that, if $I \ne J$ and $\varepsilon < \frac{\alpha}{2}$ then 
\begin{align} \label{eq:inters}
B_\varepsilon (L_I^{> \alpha}) \cap B_\varepsilon (L_J^{> \alpha}) = \emptyset.
\end{align}
Indeed, if $j \in J \backslash I$, then for any  $z  = (z_1, \ldots, z_n) \in L_{J}^{> \alpha}$ we have that $z_j = i b_j$ for $b_j \in \mathbb{R}$ and $b_j > \alpha$. On the other hand, for $w = (w_1, \ldots, w_n) \in L_I^{> \alpha}$ we have that $\Im (w_i) = 0$ so that the distance between $z$ and $w$ is greater than $\alpha$. 

Given a simplicial complex $K$ and $\alpha > \delta$, let 
\begin{align} L_{K, \delta, \alpha} := Z_\delta^n \backslash \left( \bigcup_{J \not\in K} L_J^{>\alpha} \right).
\end{align}
One notes that if $\delta \leq \delta^\prime$ and $\alpha \leq \alpha^\prime$ then $L_{K, \delta, \alpha} \subseteq L_{K, \delta^\prime, \alpha^\prime}$. We also have that
\begin{align}
L_{K, \delta, \infty} & = Z_\delta^n, & \bigcap_{\alpha > 0, \delta > 0} L_{K, \delta, \alpha} & = L_K.
\end{align}
The intuition behind this construction is that as $\alpha$ decreases from $\infty$, for each face $J$ of the $n$-simplex that does not lie in $K$, one lowers a wall corresponding to the strata $L_J$.
For $\varepsilon < \frac{\delta}{2}$ we take $U_{K, \delta, \alpha, \varepsilon}$ to be a $\varepsilon$-neighborhood of $L_{K, \delta, \alpha}$ and $h_{K} : U_{K, \delta, \alpha, \varepsilon}  \to \mathbb{R}$ kinetic energy relative to $L_{K, \delta, \alpha}$.
\begin{lemma}
	The collection $\mathbf{H} := \{h_{K|_I} : U_{K|_I, \delta, \alpha, \varepsilon} \to \mathbb{R} \}_{I \subset [m]}$ is a compatible smoothing of $\tilde{H}_K$.
\end{lemma} 
\begin{proof}
	One observes that, because equation \eqref{eq:inters}, the kinetic energy functions $h_{K|_I}$ are locally either the kinetic energy relative to the isotropic boundaries $\partial_+ L_I^{ \leq \alpha}$ or the kinetic energy relative to $Z_\delta^n$. Both of these functions are $i$-convex and that, near $L_I$, an application of Lemma~\ref{lemma:Hconvex} shows that $g_I = \|1 + \sum_{i \in I} \Re (z_i) e_i + \sum_{i \not\in I} \Im (z_i) e_i\|$ suppresses $h_{K}$ in a neighborhood of $L_I^{ \leq \alpha}$. Taking the maximum of $g_I$ and smoothing gives a global suppressing function $g$ for $h_K$ validating the partial convexity of $h_K$ and Definition~\ref{def:ham:0}.
	
	The fact that $\mathbf{H}$ is stable follows at once from the definition. The decomposable property is a result of the fact that, for any $\sigma \in K$ and $I \subset \lk (\sigma)$, $L_{K, \delta, \alpha}$ is locally a product of $L_{K|_I, \delta, \alpha}$ and $L_{\sigma, [m] \backslash I}$ near any sufficiently displaced $u_I$. The expanding property follows from considering the retraction $\pi : U_{K, \delta, \alpha, \epsilon} \to L_{K, \delta, \alpha}$ given by taking the closest point. Then, in a neighborhood of $\partial_+ L_I^{\leq \alpha}$, the flow of $h_K$ is given by the product of cogeodesic flow and rotation about $L_I^{\leq \alpha}$. Due to the convexity of $\partial_+ L_I^{\leq \alpha}$ relative to the axes, this flow is monotonic with respect to the coordinate functions. When in the neighborhood of $Z_\delta^n$, the monotonicity follows from the $A_2$-Hamiltonian case.
	
	Properties \ref{def:ham:1} and \ref{def:ham:2} follow immediately from the construction of $L_{K, \delta, \alpha}$.
\end{proof}

\subsection{The category $\wrapped{K}{\mathbf{H}}$}
In this section we fix a compatible smoothing $\mathbf{H} = \{H_I : U_I (\varepsilon) \to \mathbb{R} \}$ of $\tilde{H}_K$ and, if necessary, replace $U_I (\varepsilon)$ with the  sublevel set $U_K (\delta )= \{p \in U_K (\varepsilon ) : H_K \leq \delta \}$ for some $\delta < \varepsilon$. Since $H_K$ is partially convex, Theorem~\ref{thm:exist} gives a well defined $A_\infty$-category $\wrapped{U_K (\delta )}{H_K}$. Within this category, there are several natural objects indexed by the combinatorial data of $K$. In particular, given any strata $L_\sigma$ of $L_K$, consider the complement 
\begin{align} L^{sm}_\sigma = L_\sigma  \backslash \left(  \sing{K}  \cap L_\sigma \right) 
\end{align}  which consists of the smooth points of $L_K$ lying in $L_\sigma$. Write $L^{sm}_K$ for the union $\cup_{\sigma \in K} L^{sm}_\sigma$ of all smooth points of $L_K$. For $p$ in a component of $L_\sigma^{sm}$, and sufficiently far away from $\sing{K}$, the sets $\tilde{U}_K (\delta )$ and $U_K (\delta )$ are locally the same and are isomorphic to the $\sqrt{2\delta}$-ball bundle of $L_\sigma$ in $T^* L_\sigma$. 

Let $p \in L^{sm}_\sigma \backslash \singe_K$ and denote the orthogonal Lagrangian  to $L_\sigma$ at $p$ in $U_K (\delta )$ as
\begin{align} L^\perp_p = \left\{ p +  \sum_{i \in [m]} z_i e_i  \in U_K (\delta ): \Re (z_j ) = 0 = \Im (z_k) \text{ for all } j \in \sigma , k \not\in \sigma  \right\}.
\end{align} 
We affix a grading and the trivial Pin structure to every $L_p^\perp$ to make it an object of $\wrapped{U_K (\delta )}{H_K}$.

\begin{definition} Let $K$ be a simplicial complex and $\mathbf{H} = \{H_I\}$ a compatible smoothing. The category $\wrapped{K}{\mathbf{H}}$ is the full subcategory of $\wrapped{U_K (\delta )}{H_K}$ generated by the objects $\left< L_p^\perp : p \in L^{sm}_\sigma \backslash \singe_K, \sigma \in K \right>$.
\end{definition}
As will be shown in Section~\ref{sec:hms}, the category $\wrapped{K}{\mathbf{H}}$ admits a fully faithful inclusion into the equivariant $A$-model mirror of a quasi-affine toric variety.

We now show that the generating collection $\left< L_p^\perp : p \in L^{sm}_\sigma \backslash \singe_K, \sigma \in K \right>$ for $\wrapped{K}{\mathbf{H}}$ is unnecessarily large and can be shrunk to a well controlled finite collection. The first reduction can be made by the following elementary lemma.
\begin{lemma} \label{lemma:basicwr1}
	Let  $Z$ be a smooth, connected component of $L^{sm}_K$.
	If $p, p^\prime \in Z$ then $L^\perp_{p} \simeq L^\perp_{p^\prime}$ in $\wrsm{\Delta}$.
\end{lemma}
\begin{proof}
	Since $Z$ is connected, we can find a sequence $p = p_0, \ldots, p_m = p^\prime \in Z$ such that $p_i$ and $p_{i + 1}$ are contained in a geodesically convex neighborhood $V_i$ of $Z \subset L_\sigma$. But in every such neighborhood, the cotangent bundle ball bundle $B_\delta (V_i)  \subset T^* V_i$ of radius $\sqrt{2\delta}$ admits an  embedding  $i:(B_\delta(V_i), \partial B_\delta (V_i)) \to (\overline{U}_K (\delta ), \partial \overline{U}_K (\delta ))$. By property \ref{def:ham:2} of the compatible Hamiltonian $H_K$, we have that $H_K |_{B_\delta (V_i)}$ equals the kinetic energy relative to $V_i$. This shows that the minimal cogeodesic	$\delta \in \mathcal{X}_w (L^\perp_{p_i} , L^\perp_{p_{i +1}})$ and its negative
	$\delta^{-1} \in \mathcal{X}_w (L^\perp_{p_{i + 1}}, L^\perp_{p_i} )$ map to inverse morphisms in $\text{Hom}_{\wrapped{K}{\mathbf{H}}} (L^\perp_{p_i}, L^\perp_{p_{i + 1}} )$ and $\text{Hom}_{\wrapped{K}{\mathbf{H}}} (L^\perp_{p_{i + 1}}, L^\perp_{p_i})$ respectively.
\end{proof}
We now obtain a convenient indexing set for the smooth components of $L^{sm}_\sigma$.  For any $I \subseteq  [m]$, write $L_I$ for the Lagrangian defined in equation~\eqref{eq:stratadef} where $\sigma$ is replaced by $I$. We consider orthants of $L_I$ indexed by functions 
\begin{align} \label{eq:orthant}
\mathbf{c} : [m] \to \{\pm 1, -i\}.
\end{align} We say that such a function $\mathbf{c}$ is $I$-positive if $\mathbf{c}^{-1}(-i) = I$ and define
\begin{align*}
L_\mathbf{c} = \{(z_1, \ldots, z_n) \in L_I : \Re (\mathbf{c} (j) z_j)
> 0 \text{ for all }j \in [m]\}.
\end{align*}
We observe that $\sqcup_{\mathbf{c} \, I\text{-positive}} L_\mathbf{c}$ equals the complement of the coordinate hyperplanes in $L_I$. 
\begin{lemma} \label{lemma:basicwr2}
	The components of $L^{sm}_\sigma$ are in one to one
	correspondence with functions $f :\lk (\sigma) \to \{\pm 1\}$. 
\end{lemma}
\begin{proof}
	We must decide when two sets $L_{\mathbf{c}_1}$ and	$L_{\mathbf{c}_2}$ lie in the same component of $L_\sigma^{sm}$ for two $\sigma$-positive functions $\mathbf{c}_i$. Let $f_i$ be the restriction of each	$\mathbf{c}_i$ to $\lk (\sigma )$ and suppose $p_i \in	L_{\mathbf{c}_i}$. Let $\mathcal{P} (p_1, p_2, L_\sigma )$ be the set	of paths from $p_1$ to $p_2$ in $L_\sigma$. We observe that $p_1$ and $p_2$ lie in the same component of $L^{sm}_K$ if and only if there exists $\ell \in \mathcal{P} (p_1, p_2, L_\sigma )$ with $\text{im}(\ell ) \cap L_\tau = \emptyset$ for all $\tau \in K$ with $\tau \ne \sigma$. Indeed, $p_1$ and	$p_2$ are in the same component if and only if they are in the same path component and any path $\ell$ in $L^{sm}_\sigma$ will not intersect any other strata $L_\tau$ (since $L_\sigma  \cap L_\tau  \subset L_K \backslash L^{sm}_K$). 
	
	For $k \in [m]$, let $\pi_k$ be the projection to the $k$-th coordinate. It is clear that $\ell \in \mathcal{P} (p_1, p_2,L_\sigma )$ can be chosen so that $\pi_k (\ell (t)) \ne 0$ for all $t$ if and only if $\mathbf{c}_1 (k) = \mathbf{c}_2 (k)$. It also follows from the definition of $L_\sigma$ that $\pi_k^{-1} (0) \cap L_\sigma = L_\tau \cap L_\sigma$ if and only if $\tau = \sigma \cup \{k\} \in \lk (\sigma )$. For all $\sigma^\prime \in K \backslash \lk (\sigma )$, the intersection $L_\sigma \cap L_{\sigma^\prime}$ is either empty or has codimension greater than $1$, so that its complement in $L_\sigma$ is connected. Thus $\text{im} (\ell )$ must intersect $L_\tau$ for some $\tau \in K$ if and only if $f_1 \ne f_2$. 
\end{proof}
For $f : \lk (\sigma ) \to \{\pm 1\}$, denote the associated connected component of $L_\sigma^{sm}$ as 
\begin{align} \label{eq:lsigf} L_{\sigma, f} = \{(z_1, \ldots, z_n) \in L_\sigma : f(j) \Re (z_j) \in  \mathbb{R}_{\geq 0} \text{ for } j \in \lk (\sigma)\}. \end{align}
\begin{proposition} \label{prop:gencol}
	For each $\sigma \in K, f : \lk (\sigma	) \to \{\pm 1\}$, choose an element $ p_{\sigma , f} \in L_{\sigma , f}$. Then the collection \[\left< L^\perp_{p (\sigma, f) }  \right> \] generates $\wrapped{K}{\mathbf{H}}$.
\end{proposition}
\begin{proof} This is an immediate application of Lemmas~\ref{lemma:basicwr1} and \ref{lemma:basicwr2}.
\end{proof}
Proposition~\ref{prop:gencol} gives a convenient finite collection of generators for $\wrapped{K}{\mathbf{H}}$. The next theorem establishes the functorial behavior of $\wrapped{K}{\mathbf{H}}$.
\begin{theorem} \label{thm:functor}
	Let $\sigma \in K$ and $I \subseteq [m]$ a subset of points of $\lk (\sigma)$. Then there is a natural, fully faithful functor,
	\begin{align}
	\iota_K^{\sigma, I}: \wrapped{K|_I}{\mathbf{H}|_I} \to \wrapped{K}{\mathbf{H}}.
	\end{align}
\end{theorem} 
Here the word natural means that if $\sigma_1 \in K$, $I_1$ is a subset of points of $\lk (\sigma_1)$, $\sigma_2 \in \lk (\sigma_1)$ and $I_2$ is a subset of $\lk_{\lk_{I_1} (\sigma_1)} (\sigma_2 )$ then there is a natural equivalence between $\iota_K^{\sigma_2, I_2} \circ \iota^{\sigma_1, I_1}_{\lk_{I_2} (\sigma_2)}$ and $\iota_K^{\sigma_1, I_1}$.

\begin{proof}
	We appeal to Definitions~\ref{def:decomposable} and \ref{def:expanding} and the fact that $\mathbf{H} = \{H_I\}$ is a decomposable collection of expanding Hamiltonians. Let $u_I = \sum_{j \not\in I \cup \sigma} a_j e_j + \sum_{j \in \sigma } i b_j e_j$ be sufficiently displaced and $\iota_I : \mathbb{C}^I \to \mathbb{C}^m$ the affine inclusion sending $w$ to $u_I + w$. We also write $\iota_I$ for its restriction to the inclusion $U_{K|_I} (\delta ) \to U_{K} (\delta)$. 
	
	Suppose  $p \in L_{K|_I} \backslash \singe_{K|_I}$ and write $q = \iota_I (p) = p + u_I$. We define the functor $\iota_K^{\sigma, I}$ on generating objects by taking 
	\begin{align}
	\iota_K^{\sigma, I} (L_p^\perp ) = L_q^\perp .
	\end{align}
	In a neighborhood of $q$, property~\ref{def:ham:2} and the fact that $\mathbf{H}$ is decomposable implies that 
	\begin{align} \label{eq:hamdecomp} H_K = \tilde{H}_K = \tilde{H}_{K|_I} + \frac{1}{2} d_{L_{\sigma, [m] \backslash I}}^2 \end{align} and its Hamiltonian vector field decomposes as $X_K = X_{K|_I} + Y$ where $Y$ is the Hamiltonian vector field of $\frac{1}{2} d_{L_{\sigma, [m] \backslash I}}^2$ in $U$.  Let $\varphi_t^K$ be the Hamiltonian flow of the $H_K$. Using the fact that $\mathbf{H}$ is decomposable, we have that $\iota_I \circ \varphi^{K|_I}_t = \varphi^K_t \circ \iota_I$. Thus, suppose $p^\prime \in L_{K|_I} \backslash \singe_{K|_I}$, $q^\prime = \iota_I (p^\prime)$ and $x : [0,1] \to U_{K|_I} (\delta )$ is an integral curve in $\mathcal{X}_w (L_p^\perp , L_{p^\prime}^\perp )$ from equation~\eqref{eq:defflow}. Then $\iota_I \circ x \in \mathcal{X}_w (L_q^\perp , L_{q^\prime}^\perp )$ inducing the function
	\begin{align} \iota_K^{\sigma, I} (x) = \iota_I \circ x \end{align}
	on morphisms. Our task is thus two-fold; first we show that this assignment is full (as the faithful property is obvious) and then we demonstrate that it is compatible with the $A_\infty$-structure maps.

	Recall from Lemma~\ref{lemma:prop1} that, in a sufficiently small neighborhood $U$ of the origin in $\mathbb{C}^{[m] \backslash I}$, equation~\eqref{eq:decomp} gives the decomposition
	\begin{align} 
	L_{K} \cap \left[(\mathbb{C}^I \oplus U ) + u_I\right] = \iota_I ( L_{K|_I} ) \oplus \left(L_{\sigma, [m] \backslash I} \cap U \right).
	\end{align}
	where $L_{\sigma, [m] \backslash I}$ is defined in equation~\eqref{eq:sigbk}.  We may also decompose the orthogonal $L^\perp_q$ as the direct sum $\iota_I (L_p^\perp) \oplus i \cdot L_{\sigma, [m] \backslash I}$ (which we identify with its intersection in $U_\delta (K)$).
	
	Let $F : \mathbb{C}^m \to \mathbb{R}$ be the  function 
	\begin{align*} F (z) & =  \frac{1}{2}\left( d_{i\cdot L_{\sigma, I}} ( \pi_{[m] \backslash I} (z - u_I ) )\right)^2 \\ & = \sum_{j \in \sigma} (y_j - b_j)^2 + \sum_{j \not\in I \cup \sigma} (x_j - a_j)^2. \end{align*}
	Note that for any $q^\prime \in \iota_I (L^{sm}_{K|_I})$ we have that $F (L_{q^\prime}^\perp ) = 0$. 
	
	We claim that, for any $z \in L_q^\perp$, $(X_K F) (z) = 0$. Explicitly, for $z = \sum_{j = 1}^m (x_j + iy_j) e_j$, one computes
	\begin{align} \label{eq:yvf}
	Y_{z} = \sum_{j \in \sigma} -x_j \partial_{y_j} + \sum_{j \not\in I \cup \sigma} y_j \partial_{x_j}.
	\end{align} and 
	\begin{align*} 
	Y_z F & = \sum_{j \in \sigma} -x_j (y_j - b_j) + \sum_{j \not\in I \cup \sigma} y_j (x_j - a_j).
	\end{align*}
	Now, if $z = (z_I, z_{[m] \backslash I}) \in L_q^\perp$ then there are real constants $c_j$ for $j \not\in I \cup \sigma$ and $d_j$ for $j \in \sigma$ such that 
	\begin{align}\label{eq:zmI}
	z_{[m] \backslash I} = \sum_{j \in \sigma } (d_j + i b_j) e_j + \sum_{j \not\in I \cup \sigma} (a_j + i c_j) e_j  .
	\end{align}
	verifying $Y_z F (z)= 0$ for all $z \in L_{q}^\perp$. As $F$ is independent of the variables indexed by $I$, we also conclude $(X_K F) (z) = 0$.
	
	Now suppose $x : [0,1] \to U_\delta (K)$ is any $w X_{K}$ integral curve for $w \in \mathbb{Z}_{> 0}$ such that $x (0) \in L_q^\perp$. Then we have $(X_K F) (x(0)) = 0$ and, as $H_K$ is expanding, $(X_K F)(x(t_0)) \geq 0$ for all $t_0 \geq 0$ (this can be seen by applying the expanding condition to each summand in $F$). In particular, if $F(x(t_0) ) > 0$ for any $t_0$, then $F(x(t)) > 0$ for all $t \geq t_0$. 
	
	Using equations~\eqref{eq:yvf} and \eqref{eq:zmI} we have that if $x(0) = (z_I, z_{[m] \backslash I})$, then $F(x(t)) > 0$ for all $t > 0$ if and only if there exists $j$ such that either $c_j \ne 0$ or $d_j \ne 0$. In particular, if $\varphi^K_t : U_\delta (K) \to U_{\delta} (K)$ is the time $t$ flow of $X_K$, then 
	\begin{align} F \left( \varphi^K_t \left(L_q^\perp \backslash  \iota_I(L_p^\perp ) \right) \right) > 0 \end{align}
	for all $t > 0$. Intuitively, all elements of $L_q^\perp$ which are not in the image of $L_p^\perp$ move away from the subspace $\mathbb{C}^m \oplus i\cdot L_{\sigma, I}$. But if $p^\prime \in L^{sm}_{K|_I}$  and $q^\prime = \iota_I (p^\prime )$, then $F (L_{q^\prime}^\perp) = 0$. This implies that if $x  \in \mathcal{X}_w (L_q^\perp , L_{q^\prime}^\perp )$, then $F (x(1)) = 0$ so that $x(0) \in \iota_I (L_p^\perp )$ and $x \in \iota_K^{\sigma, I} (\mathcal{X}_w (L_p^\perp, L_{p^\prime}^\perp))$. Thus $\iota_K^{\sigma, I}$ is fully faithful on morphisms.
	
	Now suppose $p_0, \ldots, p_d \in L_{K|_I} \backslash \singe_{K|_I}$ and write $L_i$ for $L_{p_i}^\perp$ and $\tilde{L}_i$ for $L_{\iota_I (p_i)}^\perp$. Let $x^k \in \mathcal{X}_{w^k} (L_k, L_{k + 1})$ for every $0 \leq k \leq d$ so that $\mathbf{x} = \{x^0, \ldots, x^d\}$  and denote $\tilde{x}^k$ for $\iota_K^{\sigma, I} (x^k)$ and $\tilde{\mathbf{x}}$ for $\{\tilde{x}^0, \ldots, \tilde{x}^d\}$. Since $\iota_I$ is holomorphic and commutes with $\varphi_t$, composing with $\iota_I$ induces a map moduli space $\Theta_{\mathbf{x}, I} : \mathcal{R}^{d + 1, \mathbf{w}} (\mathbf{x} ) \to \mathcal{R}^{d + 1, \mathbf{w}}( \tilde{\mathbf{x}})$. Using the (local) decomposition in equation~\eqref{eq:hamdecomp} we also have that there is a neighborhood $V$ of $\iota_I (U_{K|_I} (\delta))$ such that projection $\pi_{[m] \backslash I} : V \to \mathbb{C}^{[m] \backslash I}$ is holomorphic and $(\pi_{[m] \backslash I})_* (X_K) = Y$. This implies that if $(S,u ) \in \mathcal{R}^{d + 1, \mathbf{w}} (\mathbf{x} )$ and $u^{-1} (V) = {S \backslash C} $ for some compact set in the interior of $S$, we have that $\pi_{[m] \backslash I} \circ u$ satisfies Floer's equation for  $H = \frac{1}{2} d_{L_{\sigma, I}}^2$. However, as $(\pi_{[m] \backslash I} \circ u) |_{\partial_k S} = 0$,  unique continuation for $\bar{\partial}_{\mathbf{I}_H}$ (see \cite[Thoerem~2.3.2]{ms})  implies that $\pi_{[m] \backslash I} \circ u$ is constant. Thus $C = \emptyset$ and $\Theta_{\mathbf{x},I}$ is an isomorphism, concluding the proof.
\end{proof}
We can use Theorem~\ref{thm:functor} to explicitly describe certain morphism groups and compositions in $\wrapped{K}{\mathbf{H}}$. Let $\sigma \in K$ and suppose $f_1, f_2 : \lk (\sigma ) \to \{\pm 1\}$ are any functions. We write $f_1 \preceq f_2$ if $f_2 (k) \leq f_1 (k)$ for all $k \in \lk (\sigma )$. 
\begin{lemma} \label{lemma:morphism1}
	Let $\sigma \in K$, $f_1, f_2 : \lk (\sigma ) \to \{\pm 1\}$ and $p_i \in L_{\sigma, f_i}$ for $i = 1,2$. If $f_1 \prec f_2$ then 
	\begin{align} \label{eq:hompositive}
	\Hom_{\wrapped{K}{\mathbf{H}}} (L_{p_1}^\perp , L_{p_2}^\perp ) & = \mathbb{K} \cdot x_{f_1, f_2}
	\end{align}
	and, for $f_1 \prec f_2 \prec f_3$, we have $\mu^2 (x_{f_2, f_3} \otimes x_{f_1, f_2} ) = x_{f_1, f_3}$. 
\end{lemma}
\begin{proof}
	Letting $I = \{j \in [m]: f_1 (j) \ne f_2 (j)\}$ and applying Theorem~\ref{thm:functor}, to prove equation~\eqref{eq:hompositive}, it suffices to consider $\sigma = \emptyset$ and $f_1, f_2$ the constant functions on $[m]$ with values $1$, $-1$ respectively. Let $p_1 = \sum_{j = 1}^m e_j$ and $p_2 = - p_1$.  Taking $\varepsilon$ sufficiently small, it follows from equation~\eqref{eq:lsigf} that $p_i \in L_{\emptyset, f_i} \backslash \singe_K$ so that $L^\perp_{p_i} \subset U_\emptyset \backslash \singe_K$. 
	
	By property \ref{def:ham:2}, we have that $H_K|_{U_\emptyset \backslash \singe_K} = \tilde{H}_K = 1/2 \sum y_i^2$. Consider the Hamiltonian flow $\varphi_t$ of $H_K = \tilde{H}_K$ on $U_{\emptyset} \backslash \singe_K$. This is given by cogeodesic flow
	\begin{align} \varphi_t (x_1, y_1, \ldots, x_m, y_m) = (x_1 + t y_1,y_1, \ldots, x_m + t y_m, , y_m). \end{align} \\
	\textit{Claim}: For any $t \geq 0$ and $\varphi_t ( L_{p_1}^\perp) \subset U_\emptyset \backslash \singe_K$. \\ \\
	To verify this claim, take $z = (x_1, y_1, \ldots, x_m, y_m)$ and observe that $d_{L_\emptyset} (z) \leq d_{\sigma} (z)$ for all $\sigma \in K$ if and only if $y_i \leq |x_i|$ for every $1 \leq i \leq m$. In other words, if $V_i = \{z \in \mathbb{R}^{2n}: y_i \leq |x_i| \}$ then $U_\emptyset = \cap_{i = 1}^m V_i$ and $\partial U_{\emptyset} = \cup_{i = 1}^m \left[\partial V_i \cap (\cap_{j \ne i} V_j)
	\right]$. Let $W_i = \left\{ (x_1, y_1, \ldots, x_m, y_m) : x_i  \in [-\varepsilon , \varepsilon] , y_i \in [0, \varepsilon ] \right\}$ and observe that  $\partial V_i \subset  W_i$. Moreover, one checks that $\singe_K \subset \cup_{i = 1}^m \left[ W_i \cap  (\cap_{j \ne i} V_j)  \right]$. 
	
	Now assume that $\varepsilon < 1$, $z = (1, r_1, \ldots, 1, r_m) \in L_{p_1}^\perp$ and $t \geq 0$. Observe that $d_{L_\emptyset} (\varphi_t(z)) = \sqrt{ \sum r_i^2}$ and, for any $1 \leq j \leq m$, 
	\begin{align} d_{L_{\{j\}}}
	(\varphi_t(z)) = \begin{cases}  \sqrt{ (1 + tr_j)^2 +  \sum_{i \ne j} r_i^2} & \text{ if } r_j \geq 0, \\ \sqrt{ (1 + tr_j)^2 +  \sum_{i = 1}^m r_i^2} & \text{ if } r_j < 0. \end{cases} \end{align}
	Thus $d_{L_{\{j\}}} (\varphi_t(z)) \geq d_{L_\emptyset} (\varphi_t(z))$ for all $t \geq 0$ and $\varphi_t (L_{p_1}^\perp ) \subset U_{\emptyset}$. 
	
	To see that $\varphi_t (L_{p_1}^\perp ) \cap \singe_K = \emptyset$, suppose that $z = (1, r_1, \ldots, 1, r_m) \in L_{p_1}^\perp$ and $t_0  = \min \left\{t \geq 0 : \varphi_t (z) \in \cup_{i = 1}^m \left[ W_i \cap  (\cap_{j \ne i} V_j)  \right]\right\}$. Then there exists an $i \in [m]$ for which $\varphi_{t_0} (z) \in W_i \cap   (\cap_{j \ne i} V_j)$.
	Examining the $(x_i,y_i)$-coordinates of $\varphi_{t_0} (z)$ we have that $x_i = 1 + t_0 r_i \in [-\varepsilon, \varepsilon]$ and $y_i = r_i \in [0, \varepsilon]$. But since $r_i \geq 0$, $t_0 \geq 0$ and $\varepsilon < 1$, we obtain a contradiction and verify the claim. 
	
	With the claim, we see that, for $w \geq \frac{2}{\varepsilon}$, there is a unique $w X_K$ integral curve $x (t) = (1 -  2 t, -2 / w,\ldots, 1 - 2t , -2 / w)$ lying in $\mathcal{X}_w (L^\perp_{p_1}, L^\perp_{p_2} )$. These define the single morphism $x_{f_1, f_2}$ under passage to the limit. Composition takes place within $U_\emptyset \backslash \singe_K$ by counting the unique holomorphic triangle which connects the concatenation of cogeodesics from $p_1$ to $p_2$ and $p_2$ to $p_3$ to the unique geodesic from $p_1$ to $p_3$.
\end{proof}
Lemma~\ref{lemma:morphism1} classifies the morphisms between Lagrangian orthogonals flowing in a negative direction. The following lemma partially addresses morphisms associated to the opposite direction.
\begin{lemma} \label{lemma:morphism2}
	Let $\sigma \in K$, $f_1, f_2 : \lk (\sigma ) \to \{\pm 1\}$ and $p_i \in L_{\sigma, f_i}$ for $i = 1,2$. If $f_1 \prec f_2$ and $f_2^{-1} (-1) \backslash f_1^{-1} (-1) \in \lk (\sigma )$, then 
	\begin{align}
	\Hom_{\wrapped{K}{\mathbf{H}}} (L_{p_2}^\perp , L_{p_1}^\perp ) & = 0.
	\end{align}
\end{lemma}
\begin{proof} Take $I = f^{-1}_1 (-1)$ and apply Theorem~\ref{thm:functor} (with $\sigma = f^{-1}_2 (-1)$) to reduce to the case where $\sigma = \emptyset$ and $f_1, f_2$ the constant functions on $[m]$ with values $1$, $-1$, respectively. In this case, the assumption that $f_2^{-1} (-1) \backslash f_2^{-1} (-1) = [m] \in K$ implies that the simplicial complex $K$ is an $(m - 1)$-simplex. 
	Take $p_1 = \sum_{j = 1}^m e_j$ and $p_2 = - p_1$. As the $\mathbf{H}$ is stable, the Hamiltonian for a simplex is a sum of $A_2$-Hamiltonians on each coordinate. This implies that the flow $X_K$ factors as a flow for the respective coordinate Hamiltonians. However, the flow of an $A_2$-Hamiltonian with initial condition over positive reals maintains a positive real value, implying the result.
\end{proof}

Theorem~\ref{thm:functor} also allows us to further simplify our generating collection. Suppose $\sigma \in K$ and $k \in \lk (\sigma )$ with $\tau = \sigma \cup \{k\}$. For any $\mathbf{c} : \lk (\sigma ) \backslash \{k\} \to \{\pm 1\}$ let $\mathbf{c}_\pm : \lk (\sigma ) \to \{\pm 1\}$ be the extension 
\begin{align} \mathbf{c}_\pm & = \begin{cases}
\mathbf{c} (j) & \text{ if } j \ne k, \\ \pm 1 & \text{ otherwise}.
\end{cases} \end{align} 
Take $\mathbf{c}_0 = \mathbf{c}$ on  $\lk (\sigma ) \backslash \{k\}$ and $-i$ on $k$. 
\begin{corollary} \label{cor:extriangle}
	There is an exact triangle 
	\begin{equation*}
	\begin{tikzpicture}[node distance=2cm, auto]
	\node (A) {$L_{\mathbf{c}_+}$};
	\node (B) [right of=A, node distance=4cm]{$L_{\mathbf{c}_-}$}; 
	\node (C) [below of=A, right of=A, node distance=2cm]{$L_{\mathbf{c}_0}$};
	\draw[->] (A) to node {$x_{\mathbf{c}_+, \mathbf{c}_-}$} (B);
	\draw[->] (B) to node {} (C);
	\draw[->] (C) to node {$[1]$} (A);
	\end{tikzpicture}
	\end{equation*}
\end{corollary}
\begin{proof}
	This is an immediate consequence of the fact that such an exact triangle occurs for the simplicial complex in the one dimensional Example~\ref{example:a2}. Thus taking $I = \{k\}$ and applying Theorem~\ref{thm:functor} we obtain our result.
\end{proof}

\section{\label{sec:hms}Equivariant mirror symmetry}

Pursuing a distinct and more algebraic approach, we show in the present section that there is an equivalence between subcategories of the mirrors. 


\subsection{$K$-monomial categories}

We begin by establishing some useful notation for what follows. For any $n \in \mathbb{N}$, let
$\{e_1, \ldots, e_n\}$ be the standard basis for $\mathbb{Z}^n$ and, if $I \subseteq \{1,
\ldots, n\}$, we again take $e_I := \sum_{i \in I} e_i$. We define the monomial category
$\mathcal{C}_n$ to be the $\mathbb{K}$-linear graded category with the following structure.
\begin{align}
Ob (\mathcal{C}_n) & = \{C_I : I \subseteq [n]\},  \\
\Hom_{\mathcal{C}_n}^m (C_{I_1}, C_{I_2} ) & = \begin{cases} \mathbb{K} \cdot \left< e_{I_2
	\backslash I_1} \right> & \text{ if } I_1 \subseteq I_2 \text{ and } m = 0, \\
0 & \text{ otherwise.} \end{cases} \\
e_{I_2 \backslash I_3} \circ e_{I_1 \backslash I_2} & = e_{I_1 \backslash I_3} \text{ for any
} I_1 \subseteq I_2 \subseteq I_3 .
\end{align}
For any $A_\infty$-category $\mathcal{C}$, we let $\tw (\mathcal{C}_n)$ be the category of
twisted complexes. This is also an $A_\infty$-category which is pretriangulated, so that its
derived category $H^0 (\tw (\mathcal{C}))$ is a triangulated category. 


In particular, we consider the shifted Koszul complexes 
\begin{equation} \label{eq:Koszdef}
\mathcal{K}_{I}\{ J\} = \left( \bigoplus_{I^\prime \subseteq I} C_{I^\prime \cup
	J}[|I \backslash I^\prime|]
, K_{I,J} \right)
\end{equation}
in $\tw (\mathcal{C}_n)$ where $I \cap J = \emptyset$. These complexes are objects in $\tw
(\mathcal{C})$ with differential $K_{I, J}$. To write down this differential, order $I$
with some bijection $\sgn : I
\to [|I|]$ and take
\begin{equation} K_{I, J} = \sum_{I^\prime \subseteq I} \left( \sum_{i \in I \backslash
	I^\prime} (-1)^{\sgn (i) + | I^\prime |} e_{I^\prime \cup \{i\} \backslash I^\prime } \right).
\end{equation}
One can easily show that these objects, by definition, fit into exact triangles
\begin{figure}[ht]
	\begin{tikzpicture}[node distance=2cm, auto]
	\node (A) {$\mathcal{K}_{I}\{J\}$};
	\node (B) [right of=A, node distance=4cm]{$\mathcal{K}_{I} \{J \cup\{k\} \}$}; 
	\node (C) [below of=A, right of=A, node distance=2cm]{$\mathcal{K}_{I \cup \{k\}}
		\{J\}$};
	\draw[->] (A) to node {$\mathbf{e}_k$} (B);
	\draw[->] (B) to node {} (C);
	\draw[->] (C) to node {[1]} (A);
	\end{tikzpicture}
	\caption{\label{fig:disttriangle} Distinguished triangle of shifted Koszul complexes}
\end{figure}
for any $I \cap J = \emptyset$ with $k \not\in I \cup J$. Here the map $\mathbf{e}_k$ is defined to be $\oplus_{I^\prime \subset I} e_{I^\prime \cup \{k\} \backslash I}$. We will write $\mathcal{K}_I$ for $\mathcal{K}_I \{\emptyset\}$.

We now give a variant of the notion of a fully faithful functor in order to be able to describe
our categories axiomatically. Suppose $\mathcal{C}$ is any category with finitely many objects
indexed by subsets of $[n]$ and take $\mathcal{D}$ to be any category.

\begin{definition} \label{def:Kff}
	A functor $F : \mathcal{C} \to \mathcal{D}$ will be called $\Delta$-faithful (full)  if $I \subset J$ implies
	\begin{equation*}
	F : \Hom_{\mathcal{C}} (C_I , C_J ) \to \Hom_{\mathcal{D}} (F (C_I ) , F (C_J) )
	\end{equation*}
	is injective (surjective).
\end{definition}

For the following definition, we will say that two subsets are comparable if one is contained in the other and will say a tuple $(I_1, \ldots, I_r)$ of subsets is comparable with $(J_1, \ldots, J_r)$ if either all $I_k$ are contained in $J_k$ or vice versa.
With this in mind, and the notation established above, we can now define the main category of interest.  As before, we take $K$ to be an abstract simplicial complex on the set $[n]$. 

\begin{definition} \label{def:delmon}
	An $A_\infty$-category $\mathcal{D}$ is called a $K$-monomial category if 
	\begin{enumerate} 
		\item \label{def:delmon1} There is a $\Delta$-fully faithful functor $F : \mathcal{C}_n \to \mathcal{D}$ which is
		bijective on objects.
		\item  \label{def:delmon2} For any collection $(I, J, \{k\})$ of disjoint sets with $I \in K$ and $I \cup
		\{k \} \not\in K$, the map
		\begin{equation}
		F( \mathbf{e}_k) : F ( \mathcal{K}_{I} \{ J\} )\longrightarrow F (\mathcal{K}_{I} \{ J \cup
		\{k\}\})
		\end{equation}
		is an isomorphism in $\tw (\mathcal{D})$.
		\item \label{def:delmon3}
		For any triple $I, J, L$ of disjoint subsets in $[n]$, there is a quasi-isomorphism
		\begin{align} 
		\Hom_{\mathcal{D}}^\bullet (D_{I \sqcup L} , D_{J \sqcup L} ) \approx
		\Hom^\bullet_{\mathcal{D}} (D_I , D_J ).
		\end{align}
		Furthermore, this quasi-isomorphism is natural with respect to a comparable triple $(I^\prime, J^\prime, L)$.
	\end{enumerate}
\end{definition}
Note that, using the distinguished triangle in Figure \ref{fig:disttriangle}, we can rephrase
property \ref{def:delmon2} by saying that $F (\mathcal{K}_{I \cup \{k\}} \{J\})$ is zero in
$\tw (\mathcal{D})$ for all $I \in K$ and $I \cup \{k\} \not\in K$. 
The motivation for this definition is that we can verify properties
\ref{def:delmon1}, \ref{def:delmon2} on both the $A$-model and $B$-model side of homological
mirror symmetry. This then reduces the desired homological equivalence to our main algebraic
result, Theorem \ref{thm:mainalg}, discussed below.

We fix $\mathcal{D}$ as a  $K$-monomial
category and  assume that it is a minimal $A_\infty$-category. We denote its objects as $D_I =
F (C_I )$.
\begin{lemma} \label{lemma:prop2prime}
	For any collection $(I, J)$ of disjoint sets with $I \not\in K$,  $F
	(\mathcal{K}_{I} \{ J\} )$ is a zero object in  $\tw (\mathcal{D})$.
\end{lemma}
\begin{proof}
	Utilizing property \ref{def:delmon2} and the distinguished triangle in Figure
	\ref{fig:disttriangle}, this follows from a straightforward induction argument on
	$n = \min \{|I^\prime | : I \backslash I^\prime \in K\}$.
\end{proof}

\begin{lemma} \label{lemma:notcomp}
	If $\sigma, \sigma^\prime \in K$ are not comparable, then $\Hom_{\mathcal{D}} (D_{\sigma} , D_{\sigma^\prime }) = 0$.
\end{lemma}
\begin{proof}
	The proof is by induction on $n = |\sigma | + | \sigma^\prime |$. The statement holds
	vacuously if $n = 0$ or $1$. Assume that the lemma holds for all pairs $\tilde{\sigma}, \tilde{\sigma}^\prime$ with $|\tilde{\sigma}| + | \tilde{\sigma}^\prime | \leq n$.
	
	By property \ref{def:delmon3} of Definition~\ref{def:delmon}, we may assume that $\sigma \cap \sigma^\prime = \emptyset$.
	Let $\sigma = \sigma_1 \sqcup \sigma_2$ so that $\sigma^\prime \cup \sigma_2 \in K$ and
	for every $k \in \sigma_1$, $\sigma^\prime \cup \sigma_2 \cup \{k \} \not\in K$. Choose
	one such element $k \in \sigma_1$. By property \ref{def:delmon2} of Definition
	\ref{def:delmon}, the Koszul complex $\mathcal{K}_{\sigma^\prime \cup \sigma_2 \cup \{k \}}
	\simeq 0$ in $\tw (\mathcal{D})$. We examine the morphism complex in $\tw (\mathcal{D} )$
	\begin{align} \label{eq:zero}
	0 &=  \Hom_{\mathcal{D}}^\bullet (D_\sigma , \mathcal{K}_{\sigma^\prime \cup \sigma_2 \cup
		\{k\}} ) = \bigoplus_{J \subseteq \sigma^\prime \cup \sigma_2 \cup \{k\}}
	\Hom_{\mathcal{D}}^\bullet (D_\sigma , D_J).
	\end{align}
	Note that for any $J \subseteq \sigma^\prime \cup \sigma_2 \cup \{k\}$, either  $J =
	\sigma^\prime$,    $J \subsetneq \sigma^\prime$ or $J \cap (\sigma_2 \cup \{k\} ) \ne
	\emptyset$. In the second case, $\Hom_{\mathcal{D}}^\bullet (D_\sigma , D_J ) = 0$ by property
	\ref{def:delmon1} of Definition \ref{def:delmon}. In the third case, take $L = J \cap \sigma$
	and applies property \ref{def:delmon3} of Definition~\ref{def:delmon} to obtain $\Hom_{\mathcal{D}}^\bullet (D_\sigma , D_J
	) = \Hom_{\mathcal{D}}^\bullet (D_{\sigma \backslash L} , D_{J \backslash L} )$. Assuming that
	$\sigma \ne \sigma_2 \cup \{k\}$, either $J \subsetneq \sigma$ in which case this morphism
	space is zero by property \ref{def:delmon1} of Definition \ref{def:delmon}, or $\sigma
	\backslash L$ and $J \backslash L$ are not comparable with  $|\sigma \backslash L| + |J
	\backslash L | < n$. Applying the induction hypothesis yields a zero morphism space in the
	latter instance as well. Thus when $\sigma \ne \sigma_2 \cup \{k\}$ the only summand that is
	potentially non-zero in the complex in equation \eqref{eq:zero} is $\Hom^\bullet_{\mathcal{D}}
	(D_\sigma , D_{\sigma^\prime} )$, but since the complex is acyclic, this must be zero as well.
	
	To take care of the last case where $\sigma = \sigma_2 \cup \{k\}$, consider the collection of
	all $J \subseteq \sigma \sqcup \sigma^\prime$ for which $\sigma \subseteq J$.
	Owing again to properties \ref{def:delmon1} and \ref{def:delmon3} of Definition \ref{def:delmon}, the summands in
	equation \eqref{eq:zero} corresponding to these subsets forms the exact subcomplex $\left(
	\bigwedge^* \mathbb{K}^{\sigma^\prime} , \wedge (\sum_{j \in \sigma^\prime} e_j) \right)$. Thus
	on first page of the spectral sequence associated to the multicomplex
	$\Hom_{\mathcal{D}}^\bullet (D_\sigma , D_{\sigma \sqcup \sigma^\prime})$, we are again left
	with $\Hom^\bullet_{\mathcal{D}} (D_\sigma , D_{\sigma^\prime} ) = 0$.
\end{proof}

Now let us consider the category $\tw (\mathcal{D})$. Recall that a collection
of objects $S = \{A_i \}_{i \in I}$ is said to generate a pre-triangulated $A_\infty$-category
$\mathcal{A}$ if every object in $\mathcal{A}$ can be constructed through a sequence of taking
cones, shifts and summands starting with the objects in $S$.
\begin{lemma}  \label{lemma:generate}
	The collection of objects $\{D_\sigma : \sigma \in K \}$ generates the
	category $\tw (\mathcal{D})$.
\end{lemma}
\begin{proof}
	Let $\mathcal{A} \subset \tw (\mathcal{D} )$ be the category generated by $\{D_\sigma : \sigma \in K \}$.
	For any set $I \subseteq [n]$, we let $d (I)$ be the minimal number of elements which must be removed from $I$ in order to obtain a simplex in $K$. If $d(I) = 0$, then $D_I \in \mathcal{A}$ by definition. Assume $D_J \in \mathcal{A}$ for $d(J) < d(I)$.  By the definition of $\mathcal{K}_I$ in equation~\eqref{eq:Koszdef} and the induction hypothesis, we have that all objects in the twisted complex $\mathcal{K}_I$ other than $D_I$ are in $\mathcal{A}$. But by Lemma \ref{lemma:prop2prime}, $\mathcal{K}_I$ is quasi-isomorphic to zero in $\tw (\mathcal{D})$ which implies that $D_I \in \mathcal{A}$.
\end{proof}

Now, recasting $K$ as a poset with $\sigma \preceq \sigma^\prime$ iff $\sigma \subseteq
\sigma^\prime$, we can consider the poset representation algebra $A_K$
(see \cite{gs}). The category of finitely generated modules
$\catmod{A_K}$ has irreducible projective objects $P_{\sigma}$ for all $\sigma \in
K$. We let $\mathcal{P}_K$ be the full subcategory containing these objects. As
$\mathcal{P}_K$ generates $\tw (\oplus_{\sigma \in K} P_\sigma ) = \tw
(\catmod{A_K} )$, any functor $G : \mathcal{P}_K \to \tw (\mathcal{D})$ where
$\mathcal{D}$ is a dg or $A_\infty$-category, induces a functor $D(G) : \tw (\catmod{A_K}
) \to \tw (\mathcal{D})$. 

It follows from the definition of $A_K$ that if $\mathcal{D}$ is any $K$-monomial
category, then there is a faithful functor $i_{K}: \mathcal{P}_K \to
\mathcal{D}$ which takes $P_\sigma$ to $F(C_\sigma )$. Thus every $K$-monomial category
$\mathcal{D}$ contains $\mathcal{P}_K$ as a subcategory. 

\begin{theorem} \label{thm:mainalg}
	For any  $K$-monomial category $\mathcal{D}$, $i_K$ is a full and faithful functor inducing an $A_\infty$-equivalence $D (i_K) : \tw ( \catmod{A_K} ) \to \tw ( \mathcal{D} )$.
\end{theorem}
\begin{proof} 
	By the definition of $A_K$ and of $K$-monomial categories,  $i_K$ is faithful. Also, the definition of a $\Delta$-fully faithful functor ensures that if $\sigma \subset \sigma^\prime$ for $\sigma, \sigma^\prime \in K$, then $i_K$ yields a bijection from $\text{Hom}_{\mathcal{P}_K} (P_\sigma , P_{\sigma^\prime})$ to $\text{Hom}_{\mathcal{D}} (D_\sigma , D_{\sigma^\prime})$. On the other hand, if $\sigma , \sigma^\prime \in K$ are not comparable, then the definition of $A_K$ gives us that $\text{Hom}_{\mathcal{P}_K} (P_\sigma , P_{\sigma^\prime}) = 0$ and Lemma~\ref{lemma:notcomp} ensures that $\text{Hom}_{\mathcal{D}} (D_\sigma , D_{\sigma^\prime}) = 0$. Thus $i_K$ is a fully faithful functor and gives an equivalence between $\mathcal{P}_K$ and the full subcategory of $\mathcal{D}$ consisting of the objects $\{D_\sigma : \sigma \in K  \}$. The category $\tw (\catmod{A_K})$ is generated by $\mathcal{P}_K$ and, by Lemma~\ref{lemma:generate}, the objects $\{D_\sigma : \sigma \in K  \}$ generate $\tw (\mathcal{D})$, implying $D(i_K)$ is the indicated quasi-equivalence of categories.
\end{proof}
This theorem indicates that there exists a unique $K$-monomial category,
so that verifying the properties in Definition \ref{def:delmon} uniquely determines
$\mathcal{D}$. 

\subsection{Local equivariant mirror symmetry}

In this section we consider a subcategory in both the $B$-model and $A$-model.
For any $\sigma \in K$, let 
\begin{align*} U_\sigma = \{z_i \ne 0 : i \not\in \sigma\}
\subset \mathbb{C}^n. \end{align*} Let $Y_K$ be the quasi-affine, toric variety in
$\mathbb{C}^n$ defined as
\begin{align*}
Y_K = \bigcup_{\sigma \in K} U_\sigma .
\end{align*}
Equivalently, $Y_K$ is the toric variety defined by the fan $\Sigma_K
= \{\cone (\{e_i : i \in \sigma\}) : \sigma \in K\}$ determined by $K$ with support in $\mathbb{R}^n$. 

We consider the equivariant $B$-model $D^{eq} (Y_K)$ of $(\mathbb{C}^*)^n$
equivariant sheaves on $Y_K$. For any character $\gamma \in \mathbb{Z}^n = \Hom
((\mathbb{C}^*)^n , \mathbb{C}^*)$, let $\mathcal{O} (\gamma )$ be the
associated equivariant line bundle. Let $\mathcal{D}_B$ be the full
subcategory of $D^{eq} (Y_K )$ containing the objects
\begin{align*}
\left\{ \mathcal{O} \left(\sum_{i \in I} e_i \right)
: I \subseteq [n] \right\}.
\end{align*}
Define the local equivariant $B$-model $D_{loc}^{eq} (Y_K )$ to be the
category generated by $\mathcal{D}_B$ in $D^{eq} (Y_K )$. 
\begin{lemma} \label{lemma:Bmon}
	$\mathcal{D}_B$ is a $K$-monomial category.
\end{lemma}
\begin{proof}
	The proof of this fact follows from routine application of basic results in
	toric geometry. In particular, taking the naive  functor $F : \mathcal{C} \to \mathcal{D}_B$ sending $C_I$ to $\mathcal{O} (\sum_{i \in I} e_i)$, one may utilize \cite[Proposition~2.1, Theorem~2.6]{oda} to show that $F$ is a $\Delta$-fully faithful. To see that property~\ref{def:delmon2} of Definition~\ref{def:delmon} is satisfied, it suffices to show that $\mathcal{K}_I$ is acyclic for $I \not\in K$. Considering $\mathcal{O}(\sum_{i \in I} e_i)$ as  equivariant line bundles on $\mathbb{C}^n$ instead of $Y_K$, one observes that   $\mathcal{K}_I$ is the Koszul complex which resolves the equivariant structure sheaf on the coordinate subspace
	\begin{align*}
	Z_I := \{(z_1, \ldots, z_n) : z_i = 0, i \in I \}.
	\end{align*}
	However, it is clear that $Z_I \cap U_\sigma = \emptyset$ for all $\sigma \in K$ which implies that $Z_I \cap Y_K = \emptyset$. Thus the support of $\mathcal{K}_I$ on $Y_K$ is empty and $\mathcal{K}_I$ is consequently acyclic.
\end{proof}

The motivation for naming $D^{eq}_{loc} (Y_K )$ the \textit{local}
category stems from the description of the mirror $A$-model. The $A$-model mirror to $D^{eq} (Y_K )$ is anticipated to be the category
$\wrapped{U_{K_{eq}}}{H_{K_{eq}}}$ where $U_{K_{eq}}$ is defined  following \cite{fltz}. Consider the torus $\mathbb{T}^n = \mathbb{R}^n / \mathbb{Z}^n$ and the quotient map $\pi : T^* \mathbb{R}^n \to T^* \mathbb{T}^n$. The image $\pi(L_K)$ is then a singular Lagrangian in $T^* \mathbb{T}^n$ and its pre-image $L_{K_{eq}} := \pi^{-1} (\pi (L_K))$ is as well. Alternatively, one could take $L_{K_{eq}} = \cup_{m \in \mathbb{Z}^n} (m + L_K)$ to be the union of the integral translates of $L_K$. It is clear that the function $H_K$ can be defined naturally on a neighborhood of $\pi (L_K)$ in $T^* \mathbb{T}^n$. Precomposing with $\pi$ gives the Hamiltonian  $H_{K_{eq}}$ on the neighborhood $U_{K_{eq}}$ of $L_{K_{eq}}$. 

Now, for any integral point $m \in \mathbb{Z}^n \subset \mathbb{R}^n$, take $B (m, \delta)$ to be a radius $\delta$ ball around $m$. Then, for $\delta$ less than $1 - \varepsilon$, the intersection 
\begin{align} \label{eq:Uloc} U^{loc}_m := T^*B (m, \delta) \cap U_{K_{eq}} \end{align} is symplectomorphic to the neighborhood 
\begin{align} \label{eq:locmod}
U^{loc}_{K} = B (0, \delta) \cap U_K (\varepsilon).
\end{align}
Thus the wrapped Fukaya category $\wrsm{K}$ considered in Section~\ref{sec:local} describes
the local structure of $\wrapped{U_{K_{eq}}}{H_{K_{eq}}}$ near the most singular points of $L_{{K_{eq}}}$ (i.e. the integral points). 

Now suppose $K$ is defined on $[n]$ and, for every
$\mathbf{o} : [n] \to \{\pm 1\}$, choose a point $p (\mathbf{o} ) = (p_1,
\ldots, p_n) \in \mathbb{R}^n$ satisfying $p_k \in \mathbf{o} (k) \mathbb{R}_{>
	0}$. Let $\mathcal{D}_A$ be the full subcategory of $\wrsm{K}$ containing
the objects 
\begin{align*}
\left\{L_{p (\mathbf{o} )} | \mathbf{o} : [n] \to \{\pm 1\} \right\}.
\end{align*}
We have laid the foundation for the following result.

\begin{lemma} \label{lemma:Amon}
	$\mathcal{D}_A$ is a $K$-monomial category and $\tw (\mathcal{D}_A) \simeq
	\wrsm{K}$. 
\end{lemma}

\begin{proof}
	The collection of objects $\{L_{p (\mathbf{o} )} | \mathbf{o} : [n] \to \{\pm
	1\} \} $ is indexed by $I \subseteq [n]$ by taking $\mathbf{o}_I (k)= 1$ if and
	only if $k \not\in I$ and letting $D_I = L_{p(\mathbf{o}_I)}$. It follows from
	Lemmas~\ref{lemma:morphism1} and \ref{lemma:morphism2} that this assignment on objects lifts to a $\Delta$-fully faithful functor. 
	
	For each $\sigma \in K$ and $k \in [n] \backslash \sigma$ with $\sigma
	\sqcup \{k \} \not\in K$, let $I = \sigma$ and $J \subset [n] \backslash (\sigma
	\cup \{k\})$. Let $\mathbf{c}_0, \mathbf{c}_1 :[n] \to \{\pm 1 , -i\}$ be $I$-positive functions, as defined in equation~\ref{eq:orthant}, with $\mathbf{c}^{-1}(-1) = J$ and $\mathbf{c}_1^{-1} (-1) = J \cup \{k\}$.  Take $p_0 \in L_{I, \mathbf{c}}$ and $p_1 \in L_{I, \mathbf{c}}$. Inductively applying corollary~\ref{cor:extriangle}, we have that $L_{p_0}^\perp$ and $L_{p_1}^\perp$ are quasi-equivalent to the Koszul complexes $\mathcal{K}_I\{J\}$ and $\mathcal{K}_I\{J \cup \{k\}\}$. Since $\sigma \sqcup \{k \} \not\in K$, Lemmas~\ref{lemma:basicwr1} and \ref{lemma:basicwr2} imply that the morphism $x_{\mathbf{c}_0,\mathbf{c_1}}$ is an isomorphism.
	This verifies property \eqref{def:delmon2} of Definition \ref{def:delmon} and the fact that $\mathcal{D}_A$ is a $K$-monomial category.
	
	The observation that $L_{p_0}^\perp$ is quasi-isomorphic to $\mathcal{K}_I \{J\}$, combined with Lemmas~\ref{lemma:basicwr1} and \ref{lemma:basicwr2} gives us that $\left\{ L_{p (\mathbf{o} )} | \mathbf{o} : [n] \to \{\pm 1\} \right\}$ generates $\wrsm{K}$. Thus, we may conclude $\tw (\mathcal{D}_A ) \simeq \wrsm{K} $.
\end{proof}
With these results in hand, the mirror symmetry theorem is evident.
\begin{theorem} \label{theorem:lhms}
	The categories $D^{eq}_{loc} (Y_K )$ and $\wrsm{K}$ are equivalent.
\end{theorem}
\begin{proof}
	This follows from Lemmas \ref{lemma:Bmon}, \ref{lemma:Amon} and Theorem~\ref{thm:mainalg}.
\end{proof}

Theorem~\ref{theorem:lhms} gives a local version of homological mirror symmetry for quasi-affine toric varieties $Y_K \subset \mathbb{A}^n$. To obtain a global theorem, one must takes two steps which we now sketch. First, consider the neighborhood $U_{K_{eq}}$ of the equivariant skeleton $L_{K_{eq}}$. The for $m \in \mathbb{Z}^n$, the embeddings $j_m : U_K^{loc} \to U_m^{loc} \subset U_{K_{eq}}$ from equations~\eqref{eq:Uloc} and \eqref{eq:locmod} yield a covering of $U_{K_{eq}}$. Given a collection of neighboring lattice elements $\{m_i\}_{i \in I}$, one checks that the intersection $\cap_{i \in I} U_{m_i}^{loc}$ has a partially wrapped Fukaya category whose local model is equivalent to $\mathbb{R}^{n - |J|} \times U_{K \backslash J}^{loc}$ for some subset $J \subset [n]$. This yields an $A$-model cosheaf $\mathcal{F}^{A}$ of partially wrapped Fukaya categories determined by the subbasis of open sets $\mathcal{B} = \{\cap_{i \in I} U_{m_i}^{loc} : \{m_i\} \subset \mathbb{Z}^n\}$. One defines the category of global sections, or homotopy colimit category, to be the equivariant $A$-model mirror
\begin{align}
\label{eq:Aeqdef} \wrsm{K_{eq}} = \Gamma (U_{K_{eq}}, \mathcal{F}^A).
\end{align}

There is also a $B$-model cosheaf $\mathcal{F}^B$ on $U_{K_{eq}}$. This cosheaf assigns to $U_{m}^{loc}$ the subcategory of $D^{eq} (Y_K)$ generated by the equivariant sheaves \[\left\{\mathcal{O} \left(m + \sum_{s \in S} e_s \right): S \subseteq [n] \right\}\] on $Y_K$. If the local $A$-model of $\cap_{i \in I} U_{m_i}^{loc}$ corresponds to the subcomplex $K \backslash J$, then one defines $\mathcal{F}^B  \left( \cap_{i \in I} U_{m_i}^{loc}\right)$ to be the subcategory generated by \[\left\{\mathcal{O} \left(m + \sum_{s \in S} e_s \right): S \subseteq [n] \backslash J \right\}. \] As the collection of equivariant line bundles generates $D^{eq} (Y_K)$, the equivalence
\begin{align}
\label{eq:Beqdef} \wrsm{K_{eq}} = \Gamma (U_{K_{eq}}, \mathcal{F}^B).
\end{align}

Utilizing Theorem~\ref{theorem:lhms} then yields equivariant homological mirror symmetry.
\begin{theorem}
	The categories $\wrsm{K_{eq}}$ and $D^{eq} (Y_K)$ are equivalent.
\end{theorem}

To pass to the non-equivariant setting, one observes that both cosheaves $\mathcal{F}^A$ and $\mathcal{F}^B$ come with an equivariant action of $\mathbb{Z}^n$ and that the isomorphism of cosheaves respects this action. In particular, this implies that the orbit categories  (see \cite{keller}) of the global section categories admit an equivalence
\begin{align} \label{eq:orbeq}
\text{orb}_{\mathbb{Z}^n} (\wrsm{K_{eq}}) \cong \text{orb}_{\mathbb{Z}^n} \left( D^{eq} (Y_K) \right) .
\end{align} 
Using unique path lifting, one checks that the $\text{orb}_{\mathbb{Z}^n} (\wrsm{K_{eq}})$ is equivalent to the full subcategory $\mathcal{W}_{\mathbb{T}} (K) \subset \wrapped{\pi (U_{K_{eq}})}{H_{K_{eq}}}$ in the cotangent bundle of the torus $T^* \mathbb{T}^n$ generated by transverse Lagrangians to $\pi (L_K)$. Recall that $\pi (L_K) \subset T^* \mathbb{T}^n$ is the original Lagrangian skeleton proposed in \cite{fltz} as the mirror to $Y_K$. On the $B$-model side, it is easily shown that $\text{orb}_{\mathbb{Z}^n} \left( D^{eq} (Y_K) \right) \cong D (Y_K)$. Thus one obtains homological mirror symmetry for quasi-affine toric subvarieties $Y_K \subset \mathbb{A}^n$ as the formal corollary.
\begin{corollary}
	The categories $\mathcal{W}_{\mathbb{T}} (K)$ and $D (Y_K)$ are equivalent.
\end{corollary} 

\subsection{The punctured plane}
\begin{figure}[h]
	\includegraphics{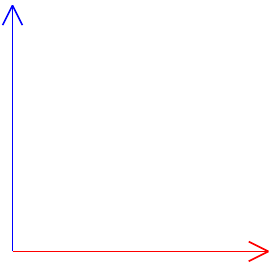}
	\caption{\label{fig:punfan} The fan for $\mathbb{C}^2 \backslash \{0\}$ }
\end{figure}
We conclude with the explicit example of the mirror of $\mathbb{C}^2 \backslash \{0\}$. The fan for $Y_K$ consists of the two rays in Figure~\ref{fig:punfan} without the positive quadrant. The simplicial set $K$ with vertex set $\{1, 2\}$ simply consists of the vertices $\{1\}, \{2\}$ and the empty set $\emptyset$. In contrast, the fan for $\mathbb{C}^2$ contains these cones and the positive quadrant and the simplicial set is an interval. There are three mirror skeleta to be considered. The first is 
\begin{align} L_K = L_\emptyset \cup L_1 \cup L_2  \subset \mathbb{R}^2 \end{align} which is the union  of the three  Lagrangian strata indexed by sets in $K$ and illustrated in Figure~\ref{fig:local}. 
\begin{figure}[h]
	\includegraphics{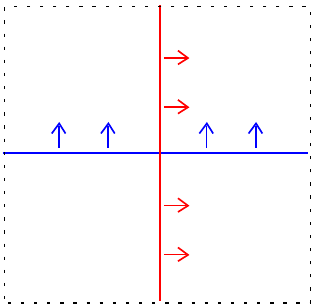}
	\caption{\label{fig:local} The local mirror Lagrangian skeleton. }
\end{figure}
In the figure, the plane itself corresponds to the zero section of $T^* \mathbb{R}^2$ and $L_\emptyset$, while the red (respectively blue) strata corresponds to $L_1$ (respectively $L_2$).  The arrows indicate the positive cotangent directions which the strata span.  The second is the equivariant skeleton $L_{K_{eq}} \subset T^* \mathbb{R}^2$ which is the union of $\mathbb{Z}^2$ translates of $L_K$ and the third is $\pi (L_K) = \pi (L_{K_{eq}})$ lying in the cotangent bundle $T^* \mathbb{T}^2$. Both of these are illustrated in Figure~\ref{fig:equivariant}.
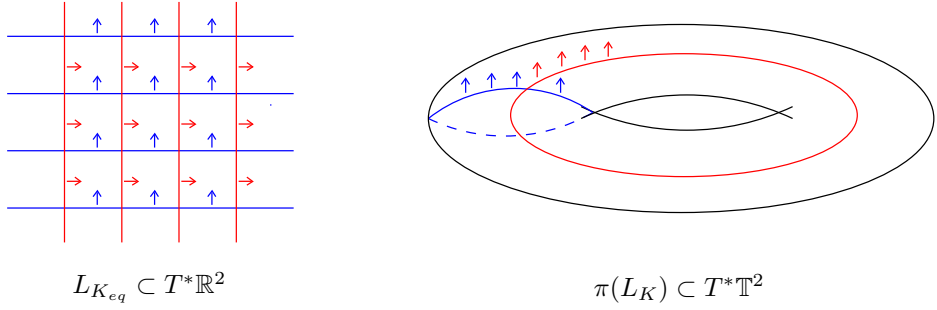
\begin{figure}[h]
	\begin{picture}(0,0)%
	\includegraphics{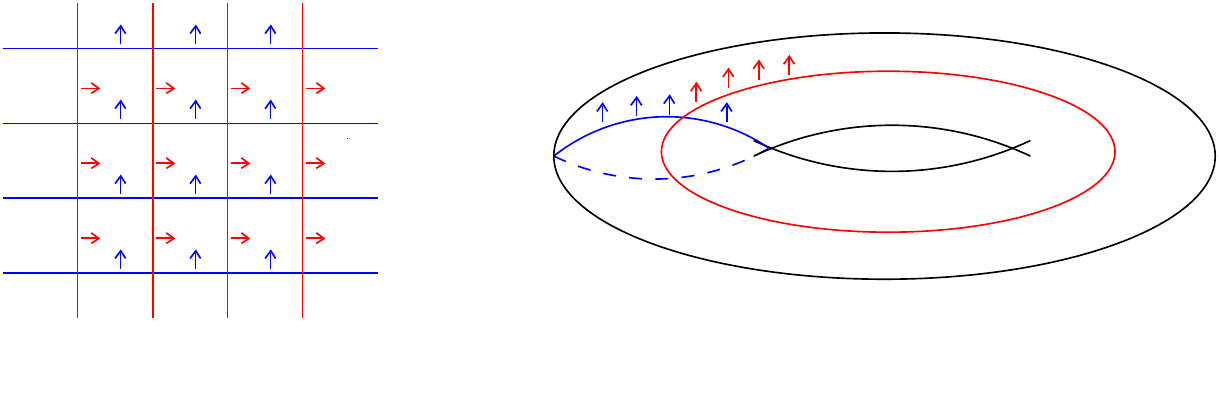}%
	\end{picture}%
	\setlength{\unitlength}{4144sp}%
	\begin{picture}(6905,2335)(3589,-4139)
	\put(7111,-4066){\makebox(0,0)[lb]{$\pi (L_K) \subset T^*\mathbb{T}^2$}}
	\put(3991,-4066){\makebox(0,0)[lb]{$L_{K_{eq}} \subset T^*\mathbb{R}^2$}}
	\end{picture}%
	\caption{\label{fig:equivariant}The equivariant and orbit skeleta.}
\end{figure}
Let us examine the category $\wrsm{K}$ by first considering the generating set from Proposition~\ref{prop:gencol}. Thus we consider the connected components of $L_K^{sm} = L_\emptyset^{sm} \sqcup L_1^{sm} \sqcup L_2^{sm}$. Using Lemma~\ref{lemma:basicwr2}, one has that $L_1^{sm}$ and $L_2^{sm}$ are connected components while $L_\emptyset^{sm}$ breaks up into the four quadrants $Q_I$ indexed by subsets $I \subseteq \{1,2\}$. Here the quadrant is given by $Q_I = \{(x_1, x_2): x_i < 0 \text{ iff } i \in I \}$. Write $D_I$ for the orthgonal Lagrangian $L_{p_I}^\perp$ where $p_I \in Q_I$. Write $E_1$ and $E_2$ for the orthogonal Lagrangians to $L_1$ and $L_2$ respectively. Lemma~\ref{lemma:Amon} gives that the category consisting of the objects $\{D_I : I \subseteq \{1,2\} \}$ is a $K$-monomial category. And by Theorem~\ref{thm:functor} for $I = \{1\}$ or $\{2\}$ we have that $E_i$ is the cone of $D_\emptyset \stackrel{\mathbf{e_i}}{\longrightarrow} D_{\{i\}}$ for $i = 1, 2$.

The $B$-model categories in this instance are the local equivariant derived category of $D^{eq}_{loc} (\mathbb{C}^2 \backslash \{0\})$ (mirror to $\wrsm{K}$), the equivariant derived category $D^{eq} (\mathbb{C}^2 \backslash \{0\})$ (mirror to $\wrsm{K_{eq}}$) and the usual derived category $D (\mathbb{C}^2 \backslash \{0\})$ of coherent sheaves (mirror to $\mathcal{W}_\mathbb{T} (K)$). We focus only the first category and observe some features of Theorem~\ref{theorem:lhms}. This category is generated by the equivariant sheaves $\mathcal{O} (0,0)$, $\mathcal{O} (1, 0 )$, $\mathcal{O} (0, 1 )$ and $ \mathcal{O} (1, 1 )$. The quiver with relations of morphisms between these objects is illustrated in Figure~\ref{fig:quiv} where the sole relation is  $e_1 e_2 = e_2 e_1$. 

\begin{figure}[h]
	\begin{tikzpicture}[cross line/.style={preaction={draw=white, -, line width=6pt}}, scale=2.3]
	\node (A) at (-.3,0) {$\mathcal{O} (0,0)$};
	\node (B) at (1,.5) {$\mathcal{O} (1, 0 )$};
	\node (C) at (1,-.5) {$\mathcal{O} (0, 1 )$};
	\node(D) at (2.3,0) {$\mathcal{O} (1, 1 )$};
	\path[->,font=\scriptsize]
	(A) edge node[above] {$e_1$} (B)
	(A) edge node[above] {$e_2$} (C)
	(D) edge node[above] {$f[1]$} (A)
	(C) edge node[above] {$e_1$} (D)
	(B) edge node[above] {$e_2$} (D);
	\end{tikzpicture}
	\caption{\label{fig:quiv} Quiver with relations for $D^{eq}_{loc} (\mathbb{C}^2 \backslash \{0\})$.}
\end{figure}
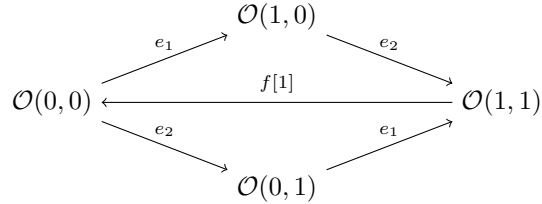
The mirror collection $D_{\emptyset}, D_{\{1\}}, D_{\{2\}}$ and $D_{\{1,2\}}$ have geodesic paths which give the morphisms mirror to $e_i$ in the quiver. However, it is by no means obvious that the morphism $f[1]$ should exist in the mirror category $\wrsm{K}$. Our approach has used the algebraic characterization of $K$-monomial categories. In this case, this amounts to observing that  $E_1$ is the cone of $D_\emptyset \stackrel{\mathbf{e_1}}{\longrightarrow} D_{\{1\}}$ and also the cone $D_{\{2\}} \stackrel{\mathbf{e_1}}{\longrightarrow} D_{\{1,2\}}$. The morphism $e_2 : E_1 \to E_1$ is an isomorphism and shows that the Koszul complex $\mathcal{K} (\{1,2\})$ is exact so that $D_{\{1,2\}}$ is isomorphic to the twisted complex $(D_\emptyset[-1] \oplus D_{\{1\}} \oplus D_{\{2\}}, \delta)$. Thus $\Hom_{\wrsm{K}} (D_{\{1,2\}}, D_\emptyset ) = \Hom_{\wrsm{K}} (D_\emptyset , D_\emptyset) [1] = \mathbb{K} [1]$ and we observe the morphism mirror to $f[1]$. 

\bibliography{mybib}{}
\bibliographystyle{plain}
\end{document}